\documentclass[12pt]{conm-p-l}
\usepackage{graphicx}
\usepackage{amsfonts}
\usepackage{amsthm}
\usepackage{amsmath}
\usepackage{amssymb}
\usepackage[arrow,matrix,curve]{xy}
\newtheorem{theorem}{Theorem}[section]
\newtheorem{corollary}[theorem]{Corollary}

\newtheorem{proposition}[theorem]{Proposition}
\theoremstyle{definition}
\newtheorem{example}[theorem]{Example}
\newtheorem{definition}[theorem]{Definition}

\newtheorem{problem}[theorem]{Problem}
\newtheorem{remark}[theorem]{Remark}
\numberwithin{equation}{section}
\newcommand{\co}{\colon\,}
\newcommand{\bT}{\mathbb T}
\newcommand{\bR}{\mathbb R}
\newcommand{\bC}{\mathbb C}
\newcommand{\bZ}{\mathbb Z}
\newcommand{\ubZ}{\underline{\mathbb Z}}
\newcommand{\bP}{\mathbb P}
\newcommand{\bH}{\mathbb H}
\newcommand{\bN}{\mathbb N}

\newcommand{\fm}{\mathfrak m}
\newcommand{\fp}{\mathfrak p}
\newcommand{\cB}{\mathcal B}
\newcommand{\cF}{\mathcal F}
\newcommand{\cH}{\mathcal H}
\newcommand{\cK}{\mathcal K}
\newcommand{\cR}{\mathcal R}

\newcommand{\cT}{\mathcal T}

\newcommand{\pt}{\text{pt}}
\newcommand{\lp}{\textup{(}}
\newcommand{\rp}{\textup{)}}

\newcommand{\Ind}{\operatorname{Ind}}
\newcommand{\Prim}{\operatorname{Prim}}
\newcommand{\Tr}{\operatorname{Tr}}

\newcommand{\Ca}{$C^*$-alge\-bra}
\newcommand{\alert}{\emph}
\newcommand{\psc}{positive scalar curvature}
\newcommand{\op}{^{\text{op}}}
\newcommand{\trans}{^{\text{t}}}
\title[Real $C^*$-Algebras]{Structure and applications\\ of real {\Ca}s}

\author{Jonathan Rosenberg}
\address{Department of Mathematics\\
University of Maryland\\
College Park, MD 20742-4015, USA} 
\email[Jonathan Rosenberg]{jmr@math.umd.edu}
\thanks{Partially supported by NSF grant DMS-1206159.  I would like to
thank Jeffrey Adams and Ran Cui for helpful discussions about Section
\ref{sec:rep} and Patrick Brosnan for helpful discussions about
Example \ref{ex:K3}.} 
\dedicatory{Dedicated to Dick Kadison, with admiration and
  appreciation}
\urladdr{http://www2.math.umd.edu/\raisebox{-3pt}{~}jmr/}
  
\begin{document}

\begin{abstract}
For a long time, practitioners of the art of operator algebras always
worked over the complex numbers, and nobody paid much attention 
to real {\Ca}s. Over the last thirty years, that situation has changed, and
it's become apparent 
that real {\Ca}s have a lot of extra structure not evident from their
complexifications. At the same time, interest in real {\Ca}s
has been driven by a number of compelling applications, for example in the 
classification of \alert{manifolds of \psc},  in
\alert{representation theory}, and in the
study of \alert{orientifold string theories}. We will discuss a 
number of interesting examples of these, and how the real Baum-Connes
conjecture plays an important role.
\end{abstract}

\maketitle

\section{Real {\Ca}s}
\label{sec:realC}

\begin{definition}
\label{def:realC}
A \emph{real {\Ca}} is a Banach $*$-algebra $A$ over $\bR$ isometrically
$*$-isomorphic to a norm-closed $*$-algebra of bounded operators on a
real Hilbert space.
\end{definition}
\begin{remark}
There is an equivalent abstract definition: a real {\Ca} is a real
Banach $*$-algebra $A$ satisfying the $C^*$-identity
$\Vert a^*a\Vert=\Vert a\Vert^2$ (for all $a\in A$) and \emph{also}
having the property that for all $a\in A$, $a^*a$ has spectrum
contained in $[0,\infty)$, or equivalently, having the property that
$\Vert a^*a\Vert\le \Vert a^*a+b^*b\Vert$ for all $a,\,b\in A$
\cite{MR0172132,MR0270162}. 
\end{remark}

Books dealing with real {\Ca}s include
\cite{MR677280,MR1267059,MR1995682}, though they all have a slightly
different emphasis from the one presented here.

\begin{theorem}[``Schur's Lemma'']
\label{thm:Schur}
Let $\pi$ be an irreducible representation of a real {\Ca} $A$ on a real
Hilbert space $\cH$. Then the commutant $\pi(A)'$ of the representation must be
$\bR$, $\bC$, or $\bH$.
\end{theorem}
\begin{proof}
Since $\pi(A)'$ is itself a real {\Ca} (in fact a real von Neumann
algebra), it is enough to show it is a division algebra over $\bR$,
since by Mazur's Theorem \cite{zbMATH02514879} (variants of
the proof are given in \cite[Theorem 3.6]{MR0172132} and \cite{MR2233367}), 
$\bR$, $\bC$, and $\bH$ are the only normed division algebras over
$\bR$.\footnote{Historical note: According to \cite{MR891589}, Mazur
  presented this theorem in Lw\'ow in 1938.  Because of space
  limitations in \emph{Comptes Rendus}, he never published the proof,
  but his original proof is reproduced in \cite{MR0448079} as well
  as in \cite{Mazet}, which also includes a copy of Mazur's original
  hand-typed manuscript, with the proof included.}  Let $x$
be a self-adjoint element of $\pi(A)'$. If $p\in 
\pi(A)'$ is a spectral projection of $x$, then  $p\cH$ and
$(1-p) \cH$ are both invariant subspaces of $\pi(A)$. By
irreducibility, either $p=1$ or $1-p=1$.  So this shows $x$ must be
of the form $\lambda\cdot 1$ with $\lambda\in \bR$. Now if
$y\in \pi(A)'$, $y^*y = \lambda\cdot 1$ with $\lambda\in \bR$, and
similarly, $yy^*$ is a real multiple of $1$.  Since the spectra of
$y^*y$ and $yy^*$ must coincide except perhaps for $0$, $y^*y=yy^*=
\lambda=\Vert y\Vert^2$ 
and either $y=0$ or else $y$ is invertible (with inverse $\Vert
y\Vert^{-2} y^*$). So $\pi(A)'$ is a division algebra.
\end{proof}
\begin{corollary}
\label{cor:3types}
The irreducible $*$-representations of a real {\Ca} can be classified
into three types: real, complex, and quaternionic. {\lp}All of these can
occur, as one can see from the examples of $\bR$, $\bC$, and $\bH$ acting on
themselves by left translation.{\rp}
\end{corollary}

Given a real {\Ca} $A$, its complexification $A_{\bC}=A+iA$ is a complex
{\Ca}, and comes with a real-linear $*$-automorphism $\sigma$ with
$\sigma^2 = 1$, namely complex conjugation (with $A$ as fixed
points). Alternatively, we can consider $\theta(a) =
\sigma(a^*)=(\sigma(a))^*$. Then $\theta$ is a (complex linear)
$*$-antiautomorphism of $A_{\bC}$ with $\theta^2=1$. Thus we can
classify real {\Ca}s by classifying their complexifications and then
classifying all possibilities for $\sigma $ or $\theta$.  This raises
a number of questions:
\begin{problem}
\label{probl1}  
Given a complex {\Ca} $A$, is it the complexification of a real {\Ca}?
Equivalently, does it admit a $*$-antiauto\-morphism $\theta$ with
$\theta^2=1$? 
\end{problem}
The answer to this in general is no.  For example, Connes
\cite{MR0435864,MR0370209} showed that there are factors not
anti-isomorphic to themselves, hence admitting no real form.  Around
the same time (ca.\ 1975), Philip Green (unpublished) observed that a stable
continuous-trace algebra over $X$ with Dixmier-Douady invariant
$\delta\in H^3(X, \bZ)$ cannot be anti-isomorphic to itself unless
there is a self-homeomorphism of $X$ sending $\delta$ to $-\delta$.
Since it is easy to arrange for this not to be the case, there are
continuous-trace algebras not admitting a real form.

By the way, just because a factor is anti-isomorphic to itself,
that does not mean it has an self-antiautomorphism of period $2$, and
so it may not admit a real form. Jones constructed an example in
\cite{MR585235}.

Secondly we have:
\begin{problem}
\label{probl2}
Given a complex {\Ca} $A$ that admits a real form, how many distinct
such forms are there? 
Equivalently, how many conjugacy classes are there of
$*$-antiautomorphisms $\theta$ with $\theta^2=1$? 
\end{problem}
In general there can be more than one class of real forms.  For
example, $M_2(\bC)$ is the complexification of two distinct real
{\Ca}s, $M_2(\bR)$ and $\bH$.  From this one can easily see that
$\cK(\cH)$ and $\cB(\cH)$, the compact and bounded operators on a
separable infinite-dimensional Hilbert space, each have two distinct
real forms.  For example, $\cK(\cH)$ is the complexification of both
$\cK(\cH_{\bR})$ and $\cK(\cH_{\bH})$. This makes the following
theorem due independently to St{\o}rmer and to Giordano and Jones
all the more surprising and remarkable.
\begin{theorem}[{\cite{MR563372,MR564145,MR728909}}]
\label{thm:Stormer}
The hyperfinite II$_1$ factor $R$ has a unique real form.
\end{theorem}
This has a rather surprising consequence: if $R_{\bR}$ is the (unique) real
hyperfinite II$_1$ factor, then $R_{\bR}\otimes_{\bR} \bH\cong R_{\bR}$
(since this is a real form of $R\otimes M_2(\bC)\cong R$).
In fact we also have
\begin{theorem}[{\cite{MR728909}}]
\label{thm:Stormer2}
The injective II$_\infty$ factor has a unique real form.
\end{theorem}

In the commutative case, there is no difference between
antiautomorphisms and automorphisms.  Thus we get the following
classification theorem.
\begin{theorem}[{\cite[Theorem 9.1]{MR0025453}}]
\label{thm:ArensKap}
Commutative real {\Ca}s are classified by pairs consisting of a
locally compact Hausdorff space $X$ and a self-homeomorphism $\tau$ of
$X$ satisfying $\tau^2=1$. The algebra associated to $(X,\tau)$,
denoted $C_0(X,\tau)$, is $\{f\in C_0(X) \mid f(\tau(x)) =
\overline{f(x)}\ \forall x\in X\}$.
\end{theorem}
\begin{proof}
If $A$ is a commutative real {\Ca}, then $A_\bC\cong C_0(X)$ for some
locally compact Hausdorff space. The $*$-antiautomor\-phism $\theta$
discussed above becomes a $*$-automorphism of $A_\bC$ (since the order
of multiplication is immaterial) and thus comes from a
self-homeomor\-phism $\tau$ of $X$ satisfying $\tau^2=1$. We recover
$A$ as 
\[
\begin{aligned}
\{f\in A_\bC \mid \sigma(f)=f\} &= \{f\in A_\bC \mid \theta(f)=f^*\} \\
 &= \{f\in C_0(X) \mid f(\tau(x)) = \overline{f(x)}\ \forall x\in X\}.
\end{aligned}
\]
In the other direction, given $X$ and $\tau$, the indicated formula
certainly gives a commutative real {\Ca}.
\end{proof}
One could also ask about the classification of real AF algebras.  This
amounts to answering Problems \ref{probl1} and \ref{probl2} for complex
AF algebras $A$ (inductive limits of if finite dimensional {\Ca}s). Since
complex AF algebras are completely classified by $K$-theory ($K_0(A)$ as
an ordered group, plus the order unit if $A$ is unital)
\cite{MR0397420}, one would expect a purely $K$-theoretic solution.
This was provided by Giordano \cite{MR931219}, but the answer is
considerably more 
complicated than in the complex case. Of course this is hardly
surprising, since we already know that even the simplest noncommutative
finite dimensional complex {\Ca}, $M_2(\bC)$, has two two distinct
real structures. Giordano also showed that his invariant is equivalent
to one introduced by Goodearl and Handelman \cite{MR904013}.
We will not attempt to give the precise statement except to say that
it involves all three of $KO_0$, $KO_2$, and $KO_4$. (For example,
one can distinguish the algebras $M_2(\bR)$ and $\bH$, both real forms
of $M_2(\bC)$, by looking at $KO_2$.)  Also, unlike the
complex case, one usually has to deal with torsion in the $K$-groups.

For the rest of this paper, we will focus on the case of separable type I
{\Ca}s, especially those that arise in representation theory.
Recall that if $A$ is a separable type I (complex) {\Ca}, with
primitive ideal space $\Prim A$ (equipped with the Jacobson topology),
then the natural map $\widehat A\to \Prim A$, sending the equivalence class
of an irreducible representation $\pi$ to its kernel $\ker \pi$, is a
bijection, and enables us to put a $T_0$ locally quasi-compact
topology on $\widehat A$.

Suppose $A$ is a complex {\Ca} with a real form. As we have seen, that
means $A$ is equipped with a conjugate-linear $*$-automorphism $\sigma$ with
$\sigma^2 = 1$, or alternatively, with a linear $*$-antiautomorphism
$\theta$ with $\theta^2 = 1$. We can think of $\theta$ as an
isomorphism $\theta\co A\to A\op$ (where $A\op$ is the opposite
algebra, the same complex vector space with the same involution $*$,
but with multiplication
reversed) such that the composite $\theta\op\circ\theta$ is the
identity (from $A$ to $(A\op)\op=A$).

Now let $\pi$ be an irreducible $*$-representation of $A$. We can
think of $\pi$ as a map $A\to \cB(\cH)$, and obviously, $\pi$ induces
a related map (which as a map of sets is exactly the same as $\pi$)
$\pi\op\co A\op \to \cB(\cH)\op$. Composing with $\theta$
and with the standard $*$-antiautomorphism $\tau\co \cB(\cH)\op \to
\cB(\cH)$ (the ``transpose map'')
coming from the identification of $\cB(\cH)$ as the
complexification of $\cB(\cH_\bR)$,
we get a composite map 
\[
\xymatrix{A\ar[r]^(.45)\theta \ar@/_1pc/_{\theta_*(\pi)}[rrr]
& A\op \ar[r]^(.4){\pi\op} & \cB(\cH)\op
  \ar[r]^{\tau} & \cB(\cH)\,.}
\]

One can also see that doing this twice brings us back where we
started, so we have seen:
\begin{proposition}
\label{prop:Ahatinv}
If $A$ is a complex {\Ca} {\lp}for our purposes, separable and type I,
though this is irrelevant here{\rp}
with a real structure {\lp}given by a
$*$-antiautomorphism $\theta$ of period $2${\rp}, then 
$\theta$ induces an involution on $\widehat A$.
\end{proposition}
\begin{proof}
The involution sends $[\pi]\mapsto [\theta_*(\pi)]$. To show that this
is an involution, let's compute $\theta_*(\theta_*(\pi))$.  By
definition, this is the composite
\[
\xymatrix@C+1ex{A\ar[r]^(.45)\theta \ar@/_1pc/[rrr]
& A\op \ar[r]^(.4){\theta_*(\pi)\op} & \cB(\cH)\op
  \ar[r]^{\tau} & \cB(\cH)}
\]
or
\[
\xymatrix{A\ar[r]^(.45)\theta  \ar@/_1pc/[rrrrr] &
  A\op\ar[r]^(.6 )\theta
  \ar@/^1.5pc/[rrr]^{\theta_*(\pi)\op}
& A \ar[r]^(.4)\pi & \cB(\cH)
  \ar[r]^(.45){\tau} & \cB(\cH)\op\ar[r]^{\tau} & \cB(\cH)},
\]
but $\theta\circ\theta$ and $\tau\circ\tau$ are each the identity, so
this is just $\pi$ again.
\end{proof}
Note that in the commutative case $A=C_0(X)$, the involution
$\theta_*$ on $\widehat A$ is just the original involution on $X$.

With these preliminaries out of the way, we can now begin to analyze
the structure of (separable) real type I {\Ca}s.  Some of this information is
undoubtedly known to experts, but it is surprisingly hard to dig it out
of the literature, so we will try give a complete treatment, without
making any claims of great originality.

The one case which \emph{is} easy to find in the literature concerns
finite-dimensional real {\Ca}s, which are just semisimple Artinian
algebras over $\bR$. The interest in this case comes from the real
group rings $\bR G$ of finite groups $G$, which are precisely of this
type. A convenient reference for the real representation theory of
finite groups is \cite[\S 13.2]{MR0450380}.  A case which is not much
harder is that of real representation theory of
compact groups.  In this case, the associated real {\Ca} is
infinite-dimensional in general, but splits as a ($C^*$-)direct
sum of (finite-dimensional)
simple Artinian algebras over $\bR$.  This case is discussed in
great detail in \cite[Ch.\ 3]{MR0252560}, and is applied to connected
compact Lie groups in \cite[Ch.\ 6 and 7]{MR0252560}.

Recall from Corollary \ref{cor:3types} that the irreducible
representations of a real {\Ca} $A$ are of three types.  How does this
type classification relate to the involution of Proposition
\ref{prop:Ahatinv} on $\widehat{A_\bC}$? The answer (which for the
finite group case appears in \cite[\S 13.2]{MR0450380}) is given as
follows:
\begin{theorem}
\label{thm:AandAC}
Let $A$ be a real {\Ca} and let $A_\bC$ be its complexification. Let
$\pi$ be an irreducible representation of $A$ {\lp}on a real Hilbert
space $\cH${\rp}. If $\pi$ is of real type, then we get an irreducible
representation $\pi_\bC$ of $A_\bC$ on $\cH_\bC$ by complexifying,
and the class of this irreducible representation $\pi_\bC$ is fixed by
the involution of Proposition \ref{prop:Ahatinv}.  If $\pi$ is of 
complex type, then $\cH$ can be made into a complex Hilbert space $\cH^c$
{\lp}whose complex dimension is half the real dimension of $\cH${\rp}
either via the action of $\pi(A)'$ or via the conjugate of this
action, and we get two distinct irreducible representations of $A_\bC$
on $\cH^c$ which are interchanged under the involution of Proposition
\ref{prop:Ahatinv} on $\widehat{A_\bC}$. Finally, if $\pi$ is of
quaternionic type, then $\cH$ can be made into a quaternionic Hilbert space
via the action of $\pi(A)'$. After tensoring with $\bC$, we get a
complex Hilbert space $\cH_\bC$ whose complex dimension is twice the
quaternionic dimension of $\cH$, and we get an irreducible
representation $\pi_\bC$ of $A_\bC$ on $\cH_\bC$ whose class is fixed
by the involution of Proposition \ref{prop:Ahatinv}.

Now suppose further that $A$ is separable and type I, so that $\pi(A)$
contains the compact operators on $\cH$, and in particular, there is
an ideal $\fm$ in $A$ which maps onto the trace-class operators. Thus
$\pi$ has a well-defined ``character'' $\chi$ on $\fm$ in the sense of
\textup{\cite{MR0147925}}, and the representations $\pi_\bC$ of
$A_\bC$ discussed above have characters $\chi_\bC$ on $\fm_\bC$. When
$\pi$ is of real type, $\chi_\bC$ restricted to $\fm$ is just $\chi$
{\lp}and is real-valued{\rp}. When $\pi$ is of quaternionic type,
$\chi_\bC$ restricted to $\fm$ is $\frac{\chi}{2}$. When $\pi$ is of complex
type, the two complex irreducible extensions of $\pi$ have characters
on $\fm$ which are non-real and which are complex conjugates of each
other, and which add up to $\chi$.
\end{theorem}
\begin{example}
\label{ex:types}
Before giving the proof, it might be instructive to give some
examples. First let $A=C^*_{\bR}(\bZ)$, the free real {\Ca} on one
unitary $u$. The trivial representation $u\mapsto 1$ is of real type
and complexifies to the trivial representation of $A_\bC=C^*(\bZ)$. 
Similarly the \emph{sign representation} $u\mapsto -1$ is of real type. The
representation $u\mapsto \begin{pmatrix} \cos \phi & \sin \phi \\
-\sin \phi & \cos \phi \end{pmatrix}$ on $\bR^2$ ($\phi$ not a
multiple of $\pi$) is of complex type. Note that this representation
is equivalent to the one given by $u\mapsto \begin{pmatrix} \cos \phi
  & -\sin \phi \\ \sin \phi & \cos \phi \end{pmatrix}$ since these are
conjugate under $\begin{pmatrix} 0&1\\1&0\end{pmatrix}$.
This representation class corresponds to a pair of inequivalent 
irreducible representations of $A_\bC=C^*(\bZ)$ on $\bC$, one given by
$u\mapsto e^{i\phi}$ and one given by $u\mapsto e^{-i\phi}$. The
involution on $\widehat{A_\bC}$ sends one of these to the other.

Next let $A=\bR Q_8$, the group ring of the
quaternion group of order $8$. This has a
standard representation on $\bH\cong \bR^4$ (sending the generators
$i,\,j,\,k \in Q_8$
to the quaternions with the same name) which is of quaternionic
type. Complexifying gives two copies of the unique $2$-dimensional irreducible
complex representation of $A_{\bC}$.

Note incidentally that $D_8$, the dihedral group of order $8$, and $Q_8$
have the same \emph{complex} representation theory.  Keeping track of
the types of representations enables us to distinguish the two groups.
\end{example}  
\begin{proof}[Proof of Theorem \ref{thm:AandAC}]
A lot of this is obvious, so we will just concentrate on the parts
that are not. If $\pi$ is of real type, its commutant is $\bR$ and its
complexification $\pi_\bC$ has commutant $\bC$ and is thus
irreducible. The class of $\pi_\bC$ is fixed by the involution, since
for $a\in A$,
\[
\theta_*(\pi_\bC)(a) = \tau\circ \pi_\bC\op\circ \theta(a) =
 \tau\circ \pi_\bC\op (a^*) = \tau(\pi\op(a^*)) = (\pi(a)\trans)\trans
 = \pi(a).
\]
If $\pi$ is of complex type, we need to show that we get
two distinct irreducible representations of $A_\bC$ which are
interchanged under the involution on $\widehat{A_\bC}$. If this were
not the case, then $\pi$ viewed as an irreducible representation on a
complex Hilbert space $\cH^c$ (via the identification of $\pi(A)'$ with $\bC$)
would extend to
an irreducible complex representation (let's call it $\pi^c$) of
$A_\bC$ which is isomorphic to $\theta_*{\pi^c}$.
Now if $a+ib\in A_\bC$, where $a,\,b\in A$, then $\sigma(a+ib)=a-ib$
and $\theta(a+ib)=a^*+ib^*$.  So
$\theta_*(\pi^c)(a+ib)=\tau\circ\pi_{\bC}\op (a^*+ib^*) =
\overline{\pi(a)} + i\overline{\pi(b)}$, since for operators on $\cH^c$,
$\tau(T^*) = \overline{T}$, the conjugate operator.
The complexification of $\cH$ is canonically identified with
$\cH^c \oplus \overline{\cH^c}$, so the complexification $\pi_\bC$ of
$\pi$ is thus identified with $\pi^c\oplus \theta_*(\pi^c)$.
If this were isomorphic to $\pi^c\oplus \pi^c$, then its commutant
would be isomorphic to $M_2(\bC)$. But the commutant of $\pi_\bC$ must
be the complexification of the commutant of $\pi$, so this is
impossible. If $\pi$ is of quaternionic type, its commutant is
isomorphic to $\bH$, which complexifies to $M_2(\bC)$.  That means the
complexification of $\pi$ has commutant $M_2(\bC)$, and thus
the complexification of $\pi$ is unitarily
equivalent to a direct sum of two copies of an irreducible
representation $\pi_\bC$ of $A_\bC$. That the class of $\pi_\bC$ is fixed by the
involution follows as in the real case.

Now let's consider the part about characters. If $\pi$ is of real
type, $\pi_\bC$ is its complexification and so the characters of
$\pi_\bC$ and of $\pi$ coincide on $\fm$.  (Complexification of
operators preserves
traces.) If $\pi$ is of quaternionic type, its complexification is
equivalent to two copies of $\pi_\bC$, so on $\fm$, the character
$\chi$ of $\pi$ is the character of the complexification of $\pi$ and so
coincides with twice the character $\chi_\bC$ of $\pi_\bC$, which is thus
necessarily real-valued. Finally, suppose $\pi$ is of complex type.
If $a\in\fm$, then $\theta_*(\pi_\bC)(a) = \overline{\pi(a)}$, so we
see in particular that the characters of $\pi^c$ and of
$\theta_*(\pi^c)$ (on $\fm$) are complex conjugates of one another,
and add up to the character of $\pi_\bC\cong \pi^c\oplus
\theta_*(\pi^c)$. But complexification of an operator doesn't change
its trace, so $\pi_\bC$ and $\pi$ have the same character on $\fm$,
and the characters of $\pi^c$ and of
$\theta_*(\pi^c)$ add up to $\chi$ on $\fm$.
\end{proof}
\begin{remark}
One can also phrase the results of Theorem \ref{thm:AandAC} in a way
more familiar from group representation theory.  Let $A$ be a real
{\Ca} and let $\pi$ be an irreducible representation of $A_\bC$ 
on a complex Hilbert space $\cH_\bC$ such that the class of $\pi$
is fixed under the involution of Proposition \ref{prop:Ahatinv}.  Then
$\pi$ is associated to an irreducible representation of $A$ of
either real or quaternionic type.  To tell which, observe that one of
two possibilities holds.  The 
first possibility is there is an $A$-invariant real
structure on $\cH_\bC$, i.e., $\cH_\bC$ is the complexification of a
real Hilbert space $\cH$ which is invariant under $A$, in which case
$\pi$ is of real type.  This condition is equivalent to saying that
there is a conjugate-linear map $\varepsilon\co \cH_\bC\to \cH_\bC$
commuting with $A$ and with $\varepsilon^2=1$.  ($\cH$ is just the
$+1$-eigenspace of $\varepsilon$.)

The second possibility is that there is a conjugate-linear map
$\varepsilon\co \cH_\bC\break\to \cH_\bC$ commuting with $A$ and with
$\varepsilon^2=-1$.  In this case if we let $i$ act on $\cH_\bC$ via
the complex structure and let $j$ act by $\varepsilon$, then since
$\varepsilon$ is conjugate-linear, $i$ and $j$ anticommute, and so
we get an $A$-invariant
structure of a quaternionic vector space on $\cH_\bC$, whose dimension
over $\bH$ is half the complex dimension of $\cH_\bC$. In this case,
$\pi$ clearly has quaternionic type.  This point of view closely
follows the presentation in \cite[II, \S6]{MR1410059}.

The books \cite{MR0252560} and \cite{MR1410059} discuss the question
of how one can tell the type (real, complex, or quaternionic) of an
irreducible representation of a compact Lie group.  In this case, one
also has a criterion based on the value of the Frobenius-Schur indicator
$\int\chi(g^2)\,dg$, which
is $1$ for representations of real type, $0$ for representations of
complex type, and $-1$ for representations of quaternionic type.  But
since this criterion is based on tensor
products for representations for groups, it doesn't seem to generalize
to real {\Ca}s in general.
\end{remark}

\section{Real {\Ca}s of continuous trace}
\label{sec:realCT}

We return now to the structure theory of (separable, say) real {\Ca}s
of type I.
\begin{definition}
Let $A$ be a real {\Ca} with complexification $A_\bC$. We say $A$ has
\emph{continuous trace} if $A_\bC$ has continuous trace in the sense
of \cite[\S4.5]{MR0458185}, that is, if elements $a\in (A_\bC)_+$ for
which $\pi\mapsto \Tr\pi(a)$ is finite and continuous on
$\widehat{A_\bC}$ are dense in $(A_\bC)_+$.
\end{definition}
\begin{theorem}
\label{thm:CTideal}
Any non-zero postliminal real {\Ca} {\lp}this is equivalent to being type I,
even in the non-separable case --- see \cite[Ch.\ 6]{MR548006} or
\cite[\S4.6]{MR1490835}{\rp} has a non-zero ideal of continuous trace.
\end{theorem}
\begin{proof}
That $A_\bC$ has a non-zero ideal $I$ of continuous trace is \cite[Lemma
  4.4.4]{MR0458185}. So we need to show that $I$ can be chosen to be
$\sigma$-invariant, or equivalently, to show that $\widehat I$ can be
chosen invariant under the involution of Proposition
\ref{prop:Ahatinv}. Simply observe that $I+\sigma(I)$ is still a
closed two-sided ideal and is clearly $\sigma$-invariant.
Furthermore, it still has continuous trace since if $a\in I_+$ and
$t_a\co\pi\mapsto \Tr\pi(a)$ is finite and continuous, then
$\pi\mapsto \Tr\pi(\sigma(a))= \Tr\pi(\theta(a))=
{\Tr\theta_*(\pi)(a)}=t_a\circ\theta_*(\pi)$ 
is also finite and continuous, so that $\sigma(a)$ is also a
continuous-trace element.
\end{proof}
\begin{corollary}
\label{cor:CTcomp}
Any non-zero postliminal real {\Ca} has a composition series
{\lp}possibly transfinite{\rp} with subquotients of continuous trace.
\end{corollary}
\begin{proof}
This follows by transfinite induction just as in the complex case.
\end{proof}

Because of Theorem \ref{thm:CTideal} and Corollary \ref{cor:CTcomp},
it is reasonable to focus special attention on real {\Ca}s with
continuous trace. To such an algebra $A$ (which we will assume is
separable to avoid certain pathologies, such as the possibility that
the spectrum might not be paracompact) is associated a 
\emph{Real space} $(X,\iota)$ in the sense of Atiyah
\cite{MR0206940}, that is, a locally compact Hausdorff space
$X=\widehat{A_\bC}$ and an involution $\iota$ on $X$ defined
by Proposition \ref{prop:Ahatinv}. The problem then arises of
classifying all the real continuous-trace algebras associated to a
fixed Real space $(X,\iota)$.  There is always a unique such
commutative real $C^*$-algebra, given by Theorem \ref{thm:ArensKap}.

When one considers noncommutative algebras, $*$-isomorphism is
too fine for most purposes, and the most natural
equivalence relation turns out to be \emph{Morita equivalence},
which works for real {\Ca}s just as it does 
for complex {\Ca}s. Convenient references for the theory of Morita
equivalence (in the complex case) are
\cite{MR679708,MR1634408}.  A Morita equivalence between real {\Ca}s
$A$ and $B$ is given by an $A$-$B$ bimodule $X$ with $A$-valued
and $B$-valued inner products, satisfying a few simple axioms:
\begin{enumerate}
\item $\langle x, y\rangle_A z = x\langle y, z\rangle_B$ and
  $\langle a\cdot x, y\rangle_B = \langle x, a^*\cdot y\rangle_B$,
  $\langle x\cdot b, y\rangle_A = \langle x, y\cdot b^*\rangle_A$
  for $x,y,z\in   X$ and $a\in A$, $b\in B$.
\item The images of the inner products are dense in $A$ and in $B$.
\item $\Vert \langle x, x\rangle_A\Vert_A^{1/2} = \Vert \langle x,
  x\rangle_B\Vert_B^{1/2} $ is a norm on $X$, $X$ is complete for this
  norm, and $A$ and $B$ act continuously on $X$ by bounded operators.
\end{enumerate}

The real continuous-trace algebras with spectrum $(X,\iota)$
have been completely classified by Moutuou \cite{MR3158706} 
up to spectrum-fixing Morita equivalence, at least in the separable
case. (Actually Moutuou worked with graded {\Ca}s.
See also \cite[\S3.3]{MR3316647} for a translation into the
ungraded case and the language we use here.)

First let us define the fundamental invariants.
\begin{definition}
\label{def:sign}
Let $A$ be a real continuous-trace algebra with spectrum
$(X,\iota)$. In other words $X=\widehat{A_\bC}$, which is Hausdorff
since $A$ has continuous trace, and let $\iota$ be the involution on
$X$ defined by Proposition \ref{prop:Ahatinv}. The \emph{sign choice}
of $A$ is the map $\alpha\co X^\iota\to \{+,-\}$ attaching a $+$ sign
to fixed points of real type and a $-$ sign to fixed points of
quaternionic type. (Of course, $\iota$ acts freely on $X\smallsetminus
X^\iota$, and the orbits of this action correspond to the pairs of conjugate
representations of complex type.)

Note that if we give $\{+,-\}$ the discrete topology, then it is easy
to see that $\alpha$ is continuous\footnote{One way to see this is to
  apply the part of Theorem \ref{thm:AandAC} about characters. If
  $e\in A_+$ is a local minimal projection near $x\in X$, then $\Tr
  \pi(e) =1$ if $\pi$ is close to $x$ and $\alpha(x)=+$ and $\Tr
  \pi(e) =2$ if $\pi$ is close to $x$ and $\alpha(x)=-$.}, so it is
constant on each connected component of $X^\iota$.

Incidentally, the name \emph{sign choice} for this invariant comes
from a physical application we will see in Section \ref{sec:KR}, where
it is related to the signs of O-planes in string theory.
\end{definition}

The other invariant of a (separable) real continu\-ous-trace algebra is the
\emph{Dixmier-Douady invariant}.  For a \emph{complex} continu\-ous-trace
algebra with spectrum $X$, this is a class in $H^2(X, \cT)$ (sheaf
cohomology), where $\cT$ is the sheaf of germs of continuous
$\bT$-valued functions on $X$.  We have a  short exact sequence of
sheaves
\begin{equation}
  0\to \bZ \to \cR \to \cT \to 1,
  \label{eq:shvs}
\end{equation}
where $\cR$ is the sheaf of germs of continuous real-valued functions,
coming from the short exact sequence of abelian groups
\[
0\to \bZ \to \bR \to \bT\to 1.
\]
Since $\cR$ is a fine sheaf and thus has no higher cohomology, the
long exact sequence in sheaf cohomology coming from \eqref{eq:shvs}
gives $H^2(X, \cT)\cong H^3(X, \bZ)$, and indeed, the Dixmier-Douady
invariant is usually presented as a class in $H^3$.

However, for purposes of dealing with \emph{real} continuous-trace
algebras, we need to take the involution $\iota$ on $X$ (and on $\cT$) into
account. This will have the effect of giving a Dixmier-Douady
invariant in an \emph{equivariant} cohomology group $H^2_\iota(X,
\cT)$ defined by Moutuou \cite{MR3141812}, who denotes it $HR^2(X,
\cT)$ (with the 
$\iota$ understood).  The $HR^\bullet$ groups are similar to, but not 
identical with, the $\bZ/2$-equivariant cohomology groups
$H^\bullet(X; \bZ/2, \cF)$ as defined in 
Grothendieck's famous paper \cite[Ch.\ V]{MR0102537}. The precise
relationship in our situation is as follows:
\begin{theorem}
\label{thm:MoutvsGroth}
Let $(X,\iota)$ be a second-countable locally compact Real space,
i.e., space with an involution, and let $\pi\co X\to Y$ be the
quotient map to $Y=X/\iota$. Then if $\cT$ is the sheaf of germs
of $\bT$-valued continuous functions on $X$, equipped with the
involution induced by the involution $(x,z)\mapsto
(\iota(x),\overline{z})$ on $X\times \bT$, then Moutuou's $HR^2(X,
\cT)$ coincides with $H^2(Y, \cT^\iota)$, where 
$\cT^\iota$ is the induced sheaf on $Y$, i.e., the sheafification of
the presheaf $U\mapsto C(\pi^{-1}(U), \bT)^\iota$.  By
\cite[(5.2.6)]{MR0102537}, there is an edge homomorphism
$H^2_\iota(X, \cT)\to H^2(X; \bZ/2, \cT)$ {\lp}which is not
necessarily an isomorphism{\rp}.
\end{theorem}
\begin{proof}
  In order to deal with quite general topological groupoids,
  Moutuou's definition of $H^\bullet_\iota(X, \cF)$ in
  \cite{MR3141812} uses
  simplicial spaces and a \v Cech construction.  But in our situation,
  $X$ and $Y$ are paracompact and the groupoid structure on $X$ is
  trivial, so by the equivariant analogue of \cite[Theorem
    1.1]{MR2231869} and the isomorphism between \v Cech cohomology and
  sheaf cohomology for paracompact spaces \cite[Th\'eor\`eme
    II.5.10.1]{MR0345092}, it reduces here to ordinary sheaf
  cohomology. 
\end{proof}

Grothendieck's equivariant cohomology groups
$H^\bullet(X; \bZ/2, \cF)$
are the derived functors of the equivariant section 
functor $X\mapsto \Gamma(X,\cF)^\iota$. Moutuou's groups are 
generally smaller. A few examples will clarify the notion, and also
explain the difference between Grothendieck's  functor and Moutuou's.
\begin{enumerate}
\item If $\iota$ is trivial on $X$, the involution on $\cT$ is just
  complex conjugation, and $H^\bullet_\iota(X,\cT)$ can be
  identified with $H^\bullet(X,\cT^\iota)=H^\bullet(X,\bZ/2)$.  Note,
  for example, that if
  $X$ is a single point, then Grothendieck's $H^\bullet(X; \bZ/2,
  \cT)$ would be the group cohomology $H^\bullet(\bZ/2, \bT)$, which is
  $\bZ/2$ in every even degree, whereas Moutuou's
  $H^\bullet_\iota(\pt, \cT)$ is just 
  $H^\bullet(\pt, \bT^\iota) = \bZ/2$, concentrated in degree $0$.
\item If $\iota$ acts freely, so that $\pi\co X\to Y$ is a $2$-to-$1$
  covering map, $H^\bullet_\iota(X,\cT)$ can be identified with Grothendieck's
  \[ H^{\bullet+1}(X; \bZ/2, \bZ)\cong
  H^{\bullet+1}(Y,\ubZ)\] for $\bullet>0$, via the
  equivariant version of the long exact sequence associated to
  \eqref{eq:shvs} and \cite[Corollaire 3, p.\ 205]{MR0102537}.
  Here   $\ubZ$ is a locally constant sheaf
  obtained by dividing   $X\times \bZ$ by the involution sending
  $(x,n)$ to $(\iota x, -n)$. 
\end{enumerate}
\begin{definition}
\label{def:DDinv}
Now we can explain the definition of the real Dixmier-Douady
invariant of a separable real {\Ca} $A$.
Without loss of generality, we can tensor $A$ with
$\cK_\bR$, which doesn't change the algebra up to spectrum-fixing
Morita equivalence.  Then $A_\bC$ becomes stable, and is locally, but not
necessarily globally, isomorphic to $C_0(X, \cK)$.  By paracompactness
(here we use separability of $A$, which implies $X$ is second countable and
thus paracompact), there is a locally finite covering $\{U_j\}$ of $X$
such that $A_\bC$ is trivial over each $\{U_j\}$. We can also assume
each $U_j$ is $\iota$-stable.  The trivializations of $A_\bC$ over
the $U_j$ give a \v Cech cocycle in $H^1(\{U_j\}, \mathcal{PU})$,
given by the ``patching data'' over overlaps $U_j\cap U_k$.  Here
$\mathcal{PU}$ is the sheaf of germs of $PU$-valued continuous
functions, since $PU$ is the automorphism group of $\cK$.
The image of this class in $H^1(X,\mathcal{PU})\cong H^2(X,\cT)$ is
the complex Dixmier-Douady invariant. Here we use the long exact
cohomology sequence associated to the sequence of sheaves
\begin{equation}
  1\to \cT \to \mathcal{U} \to \mathcal{PU}\to 1,
  \label{eq:shvs1}
\end{equation}
where again the middle sheaf is fine since the infinite unitary group
(with the strong or weak operator topology) is contractible.

In our situation, there is a little more structure because $A_\bC$ was
obtained by complexifying $A$. So we have the conjugation $\sigma$ on
$A_\bC$, which induces the involution $\iota$ on $X$ and on the sheaves
$\mathcal{U}$, $\mathcal{PU}$, and $\cT$ over $X$. Furthermore, the
cocycle of the patching data must be $\iota$-equivariant, and so
defines the real Dixmier-Douady invariant, which is its coboundary in
$H^2_\iota(X,\cT)$.
\end{definition}
\begin{theorem}[{Moutuou \cite{MR3158706}}]
\label{thm:Mout}  
The spectrum-fixing Morita equivalence classes of real
continuous-trace algebras over $(X,\iota)$ form a group 
{\lp}where the group operation comes from tensor product over $X${\rp}
which is isomorphic to $H^0(X^\iota, \bZ/2) \oplus H^2_\iota(X,\cT)$
via the map sending an algebra $A$ to the pair consisting of its sign
choice and real Dixmier-Douady invariant {\lp}in the sense
of Definitions \ref{def:sign} and \ref{def:DDinv}{\rp}.
\end{theorem}
\begin{remark}
The formulation of this theorem in \cite{MR3158706} looks rather
different, for a number of reasons, though it is actually more general.
For an explanation of how to
translate it into this form, see \cite[\S3.3]{MR3316647}.
\end{remark}
\begin{example}
\label{ex:K3}
Here are three examples, that might be relevant for physical applications,
that show how one computes the Brauer group of Theorem \ref{thm:Mout}
in practice. In all cases we will take $X$ to be a K3-surface
(a smooth simply connected complex projective algebraic surface with trivial
canonical bundle) and the involution $\iota$ to be holomorphic (algebraic).
\begin{enumerate}
  \item Suppose the involution $\iota$ is holomorphic and free. In
    this case the quotient $Y=X/\iota$ is an Enriques surface (with
    fundamental group $\bZ/2$) and $\iota$ reverses the sign of a
    global holomorphic volume form. (See for example \cite[\S1]{MR2137825}.)
    There is no sign choice invariant
    since the involution is free, and thus all representations must be
    of complex type. The Dixmier-Douady invariant lives in (twisted)
    $3$-cohomology of the quotient space $Y$. By Poincar\'e duality,
    $H^3(Y,\ubZ)\cong H_1(Y,\ubZ)$, but since $X$ is $1$-connected,
    the classifying map $Y\to B\bZ/2=\bR\bP^\infty$ is a
    $2$-equivalence (an isomorphism on $\pi_1$ and surjection on $\pi_2$)
    and induces an isomorphism on twisted $H_1$. So $H^3(Y,\ubZ)\cong
    H_1(Y,\ubZ)     \cong H_1(B\bZ/2, \ubZ) \cong
    H_1^{\text{group}}(\bZ/2, \ubZ)=0 $. So the Dixmier-Douady
    invariant is always trivial in this case.
  \item If $\iota$ is a so-called Nikulin involution (see
    \cite{MR728142,MR2274533}), then $X^\iota$  consists of $8$
    isolated fixed points. Let $Z = (X\smallsetminus X^\iota)/\iota$.
    By transversality, the complement $X\smallsetminus X^\iota$ of the
    fixed-point set is still simply connected, so $\pi_1(Z)\cong
    \bZ/2$ and the map $Z\to B\bZ/2$ is a $2$-equivalence.  We have
    $H^2_{\iota,c}(X\smallsetminus X^\iota, \cT) \cong H^3_c(Z,
    \ubZ)$, and by Poincar\'e duality, $H^3_c(Z,\ubZ)\cong H_1(Z,\ubZ)
      \cong H_1(B\bZ/2, \ubZ) = 0$. From the long exact sequence
\begin{multline}
\label{eq:sheaffixedpts1}
H^1(X^\iota, \bZ/2)=0 \to H^2_{\iota,c}(X\smallsetminus X^\iota, \cT)
\\ \to H^2_\iota(X,\cT) \to H^2(X^\iota, \cT) = 0,
\end{multline}
    (see \cite[Th\'eor\`eme II.4.10.1]{MR0345092})
    we see that $H^2_\iota(X,\cT)=0$ and the Dixmier-Douady invariant
    is always trivial in this case. However, there are many
    possibilities for the sign choice since $H^0(X^\iota, \bZ/2) \cong
    (\bZ/2)^8$. 
  \item It is well known that there are K3-surfaces $X$
    with a holomorphic map $f\co X\to \bC\bP^2$ that is a two-to-one
    covering branched over a curve $C\subset \bC\bP^2$ of degree $6$
    and genus $10$. Such a surface $X$ admits a holomorphic involution
    $\iota$ having $C$ as fixed-point set.  We want to compute the
    Brauer group of real continuous-trace algebras over $(X,\iota)$.
    Since $X^\iota = C$ is connected, there are only two possible sign
    choices, and algebras with sign choice $-$ are obtained from those
    with sign choice $+$ simply by tensoring with $\bH$. So we may
    assume the sign choice on the fixed set is a $+$. The calculation
    of the possible Dixmier-Douady invariants is complicated and uses
    Theorem \ref{thm:MoutvsGroth}.
\end{enumerate}
\end{example}
\begin{theorem}
\label{thm:DDK3}
Let $X$ be a K3-surface and $\iota$ a holomorphic involution on $X$ with
fixed set $X^\iota=C$ a smooth projective complex curve of genus
$10$ and with quotient space $Y=X/\iota=\bC\bP^2$.  Then $H^2_\iota(X,\cT)=0$.
\end{theorem}
\begin{proof}
By Theorem \ref{thm:MoutvsGroth}, $H^2_\iota(X,\cT)\cong H^2(\bC\bP^2, \cF)$,
where the sheaf $\cF$ is $\bZ/2$ over $C$ and is locally isomorphic to
$\cT$ over the complement.  By \cite[Th\'eor\`eme
II.4.10.1]{MR0345092}, we obtain an exact sequence 
\begin{multline}
\label{eq:sheaffixedpts}
H^1(C, \bZ/2) \to H^2_{\iota,c}(X\smallsetminus C, \cT) \to  
H^2_\iota(X,\cT) \to H^2(C, \bZ/2) \\ \to H^3_{\iota,c}(X\smallsetminus C,
\cT) \to H^3_\iota(X,\cT) \to H^3(C, \bZ/2)=0.
\end{multline}  
Note that $(X\smallsetminus X^\iota)/\iota\cong \bC\bP^2\smallsetminus C$.
Thus in \eqref{eq:sheaffixedpts},
$H^j_{\iota,c}(X\smallsetminus C, \cT)\cong
H^{j+1}_c(\bC\bP^2\smallsetminus C, \ubZ)\cong
H_{3-j}(\bC\bP^2\smallsetminus C, \ubZ)$. Since $C\subset \bC\bP^2$ is
a curve of degree $6$, the map $H^2(\bC\bP^2, \bZ) \to H^2(C, \bZ)$
induced by the inclusion is multiplication by $6$, and we find from
the long exact sequence
\[
H^2(\bC\bP^2, \bZ)\xrightarrow{6} H^2(C, \bZ) \to
H^3_c(\bC\bP^2\smallsetminus C, \bZ) \to 0
\]
that
$H^3_c(\bC\bP^2\smallsetminus C, \bZ)\cong H_1(\bC\bP^2\smallsetminus
C, \bZ) \cong \bZ/6$. This implies that for $j\le 1$,
$H_j(\bC\bP^2\smallsetminus C, \ubZ)$ will coincide with the
$\ubZ$-homology of a lens space with fundamental group $\bZ/6$,
or with $H_j^{\text{group}}(\bZ/6, \ubZ)= \bZ/2$, $j$ even, and $0$,
$j$ odd.  Hence \eqref{eq:sheaffixedpts} reduces to
\[
0\to H^2_\iota(X,\cT) \to \bZ/2 \xrightarrow{\delta} \bZ/2 \to
H^3_\iota(X,\cT) \to 0, 
\]
and $H^2_\iota(X,\cT)$ is either $0$ or $\bZ/2$, depending on whether
the connecting map $\delta$ is nontrivial or not.

To complete the calculation, we use Theorem \ref{thm:MoutvsGroth}.
This identifies $H^2_\iota(X,\cT)$ with $\mathrm{I}_2^{2,0}$ in the spectral
sequence 
\[
\mathrm{I}_2^{p,q}=H^p(Y, H^q(\bZ/2, \cT))\Rightarrow H^{p+q}(X;
\bZ/2, \cT)
\]
of \cite[Th\'eor\`eme 5.2.1]{MR0102537}. We will examine
this spectral sequence as well as the other one in that theorem,
\[
\mathrm{II}_2^{p,q}=H^p(\bZ/2, H^q(X, \cT))\Rightarrow H^{p+q}(X;
\bZ/2, \cT). 
\]
First consider $\mathrm{I}_2^{\bullet,\bullet}$. We have a short exact sequence
of sheaves
\begin{equation}
\label{eq:sheaves}
1 \to (\cT)_{X\smallsetminus C} \to \cT \to (\cT)_C \to 1,
\end{equation}
and $\iota$ acts trivially on $C$ and freely on $X\smallsetminus C$.
Thus $H^q(\bZ/2, (\cT)_{X\smallsetminus C}) = 0$ for $q>0$
\cite[Corollaire 3, p.\ 205]{MR0102537}. So from the long exact
cohomology sequence derived from \eqref{eq:sheaves},
$H^q(\bZ/2, \cT) = H^q(\bZ/2, (\cT)_C)$ is supported on $C$ for $q>0$. 
On $C$, the action of
$\iota$ is by complex conjugation, and so one easily sees that
$H^q(\bZ/2, \cT)= H^q(\bZ/2, (\cT)_C) = (\bZ/2)_C$ for $q>0$ even, $0$
for $q$ odd.  So for $q>0$, $\mathrm{I}_2^{p,q}$ vanishes for $q$ odd and is
$H^p(C,\bZ/2)$ for $q$ even, which is $\bZ/2$ for $p=0$ or $2$,
$(\bZ/2)^{20}$ for $p=1$, and $0$ for $p>2$. In 
particular, $\mathrm{I}_2^{0,1}=0$, so $d_2\co  \mathrm{I}_2^{0,1} \to
\mathrm{I}_2^{2,0}$ vanishes and so the edge homomorphism
$H^2_\iota(X,\cT) \to H^2(X; \bZ/2, \cT)$ is injective.  Furthermore,
$\mathrm{I}_2^{1,1}=0$ and $\mathrm{I}_2^{0,2}=\bZ/2$, so $H^2(X;
\bZ/2, \cT)$ is finite 
and
\[
\left\vert H^2(X; \bZ/2, \cT)\right\vert \le 2\cdot \left\vert
H^2_\iota(X,\cT)\right\vert. 
\]
Equality will hold if and only if the map $d_3\co 
\mathrm{I}_3^{0,2}=\bZ/2 \to \mathrm{I}_3^{3,0} = H^3_\iota(X,\cT)$ is trivial. 

Now consider the other spectral sequence
$\mathrm{II}_2^{\bullet,\bullet}$. We have $H^0(X, \cT)=C(X,\bT)$,
which since $X$ is simply connected fits into an exact sequence
\begin{equation}
\label{eq:CXT}
0 \to \bZ \to C(X,\bR) \to C(X,\bT) \to 1.
\end{equation}
Now in the exact sequence \eqref{eq:shvs}, the action of $\bZ/2$ is
via a combination of the involution $\iota$ on $X$ and complex
conjugation, which corresponds to multiplication by $-1$ on $\cR$ 
and $\bZ$.  Thus as a $\bZ/2$ module, the group on the left in
\eqref{eq:CXT} is really $\ubZ$, $\bZ$ with the non-trivial action.
On the $q=0$ row, we use equation \eqref{eq:CXT} and the
fact that higher cohomology of a finite group with coefficients in a
real vector space has to vanish to obtain that
\[
\mathrm{II}_2^{p,0}= H^p(\bZ/2, C(X,\bT))
\cong H^{p+1}(\bZ/2, \ubZ) \cong 
\begin{cases}0, &p\text{ odd},\\ \bZ/2, &p\text{ even}>0.
\end{cases}
\]
For $q>0$, we know that $H^q(X, \cT)\cong H^{q+1}(X, \bZ)$, which will
be nonzero (and torsion-free) only for $q=1$ and $q=3$.
Again, since the action on $\bZ/2$ on the constant sheaf $\bZ$ in
\eqref{eq:shvs} is by multiplication by $-1$, the action of $\bZ/2$ on
$H^3(X, \cT)\cong H^4(X, \bZ)\cong \bZ$ is by multiplication by $-1$.
The case of $H^1(X, \cT)\cong H^2(X, \bZ)\cong \bZ^{22}$ is more
complicated because we also have the action of $\iota$ on
$H^2(X,\bZ)$, which has fixed set of rank $1$ \cite[p.\ 595]{MR2137825}.
So we need to determine the structure of $H^2(X,\bZ)$ as a
$\bZ/2$-module.  The action of $\iota$ on $H^2$ has to respect the intersection
pairing, with respect to which $H^2(X,\bZ)$ splits (non-canonically)
as $E_8\oplus E_8\oplus H\oplus H\oplus H$, where $E_8$ is the $E_8$
lattice and $H$ is a hyperbolic plane ($\bZ^2$ with form given by
$\begin{pmatrix} 0&1\\1&0\end{pmatrix}$). Since $H^2(X,\bZ)^\iota\cong
\bZ$, one can quickly see that the only possibility is that $\iota$
acts by $-1$ on both $E_8$ summands and on two of the $H$ summands,
and interchanges the generators of the other $H$ summand.  Our action
here of $\bZ/2$ is reversed from this, so as  $\bZ/2$-module,
$H^1(X, \cT)\cong \bZ^{20}\oplus H$.  Since $H^p(\bZ/2, \bZ)$ is
non-zero only for $p$ even and $H^p(\bZ/2, H)=0$ for $p>0$ (by simple
direct calculation, or else by Shapiro's Lemma, since $H$ as a
$\bZ/2$-module is induced from $\bZ$ as a module for the trivial
group), we find that $\mathrm{II}_2^{1,1} = H^1(\bZ/2, \bZ^{20}\oplus
H) = 0$.  Since we already computed that $\mathrm{II}_2^{0,2} = 0$ and
$\mathrm{II}_2^{2,0} = \bZ/2$, we see that $|H^2(X; \bZ/2, \cT)|\le
2$.  It will be $0$ only if $d_2\co \mathrm{II}_2^{0,1}\to
\mathrm{II}_2^{2,0}$ is non-zero.  Putting everything together, we
finally see that the only possibilities for the two spectral sequences
are as in Figures \ref{fig:SS1} and \ref{fig:SS2}.  Comparing the
(dotted) diagonal lines of total degrees $2$ and $3$ in the two figures, we
conclude that $H^2_\iota(X,\cT)$ and $H^3_\iota(X,\cT)$ must both vanish.
\end{proof}
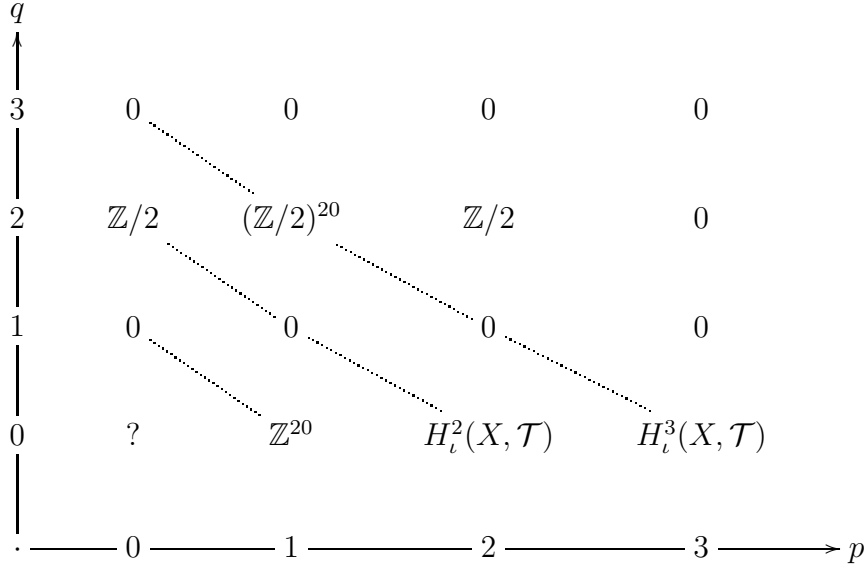
\begin{figure}[hbpt]
\[
\xymatrix{
q &&&&&\\
3\ar[u]& 0& 0& 0& 0 &\\
2\ar@{-}[u]& \bZ/2 & (\bZ/2)^{20} \ar@{.}[ul]& \bZ/2 & 0 &\\
1\ar@{-}[u]& 0& 0\ar@{.}[ul]& 0\ar@{.}[ul]& 0 &\\
0\ar@{-}[u]& ? & \bZ^{20} \ar@{.}[ul]
& H^2_\iota(X,\cT) \ar@{.}[ul]& H^3_\iota(X,\cT) \ar@{.}[ul]& \\
\cdot \ar@{-}[u]\ar@{-}[r]&  0 \ar@{-}[r]& 1 \ar@{-}[r]& 2 \ar@{-}[r]&
3\ar[r] & p} 
\]
\caption{The first Grothendieck spectral sequence
  $\mathrm{I}^{\bullet,\bullet}_2$}  
\label{fig:SS1}
\end{figure}
\begin{figure}[hbpt]
\[
\xymatrix{
q &&&&&\\
3\ar[u]& 0& \bZ/2& 0& \bZ/2 &\\
2\ar@{-}[u]& 0 & 0 \ar@{.}[ul] & 0 & 0 &\\
1\ar@{-}[u]& \bZ^{20}& 0\ar@{.}[ul]&(\bZ/2)^{20}\ar@{.}[ul]& 0& \\
0\ar@{-}[u]& ? &  0\ar@{.}[ul]& \bZ/2\ar@{.}[ul] & 0\ar@{.}[ul] & \\
\cdot \ar@{-}[u]\ar@{-}[r]&  0 \ar@{-}[r]& 1 \ar@{-}[r]& 2 \ar@{-}[r]&
3\ar[r] & p} 
\]
\caption{The second Grothendieck spectral sequence
  $\mathrm{II}^{\bullet,\bullet}_2$}  
\label{fig:SS2}
\end{figure}
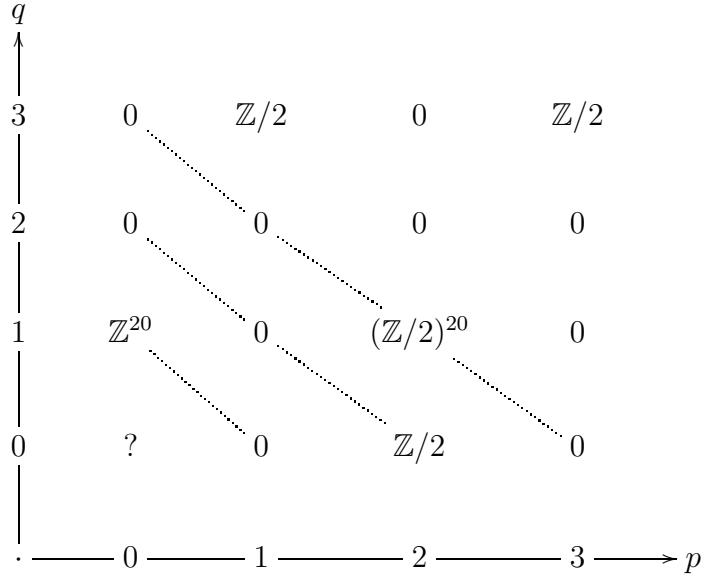

\section{$K$-Theory and Applications}
\label{sec:KR}

In this last section, we will briefly discuss the (topological)
$K$-theory of real {\Ca}s, and explain some key applications to
manifolds of {\psc} and to orientifold string theories in physics. We
should mention that other physical applications have appeared in the
theory of topological insulators in condensed matter theory
\cite{MR2806137,Kitaev},
though we will not go into this area here.  Along the way, connections
will show up with representation theory via the real Baum-Connes conjecture.

The (topological) $K$-theory of real {\Ca}s is of course a special
case of topological $K$-theory of real Banach algebras.  As such it
has all the usual properties, such as homotopy invariance and Bott
periodicity of period $8$. A convenient reference is \cite{MR1267059}.

A nice feature of the $K$-theory of real continuous-trace {\Ca}s is
that it unifies all the variants of topological $K$-theory (for
spaces) that have appeared in the literature.  This includes of course
real $K$-theory $KO$, complex $K$-theory $K$, and symplectic
$K$-theory $KSp$, but also Atiyah's ``Real'' $K$-theory $KR$
\cite{MR0206940}, Dupont's symplectic analogue of $KR$
\cite{MR0254839}, sometimes 
called $KH$, and the self-conjugate $K$-theory $KSC$ of Anderson
and Green \cite{MR0164347,Anderson}.  $KR^\bullet(X,\tau)$ is the
topological $K$-theory of the commutative real {\Ca} $C_0(X, \tau)$ of
Theorem \ref{thm:ArensKap}.  $KSC^\bullet(X)$ is $KR^\bullet(X\times
S^1)$, where $S^1$ is given the (free) antipodal involution
\cite[Proposition 3.5]{MR0206940}.  In addition, the $K$-theory of a
stable real continuous-trace with a sign choice (but vanishing
Dixmier-Douady invariant) is ``$KR$-theory with a sign choice'' as
defined in \cite{MR3267662}, and the $K$-theory of a stable real
continuous-trace with no sign choice but a nontrivial Dixmier-Douady
invariant is what has generally been called ``twisted $K$-theory'' (of
either real or complex type, depending on the types of the
irreducible representations of the algebra).  See
\cite{MR679694,MR1018964} for some of the original treatments, as well
as \cite{MR2172633,MR2513335} for more modern approaches.

\subsection{Positive scalar curvature}
\label{sec:psc}

A first area where real {\Ca}s and their $K$-theory plays a
significant role is the classification of manifolds of {\psc}.
The first occurrence of real {\Ca}s in this area is implicit in an
observation of Hitchin \cite{MR0358873}, that if $M$ is a compact
Riemannian spin manifold of dimension $n$ with {\psc}, then the
$KO_n$-valued index of the Dirac operator on $M$ has to vanish.  For
$n$ divisible by $4$, this observation was not new and goes back to
Lichnerowicz \cite{MR0156292}, but for $n\equiv 1,2$ mod $2$, a new torsion
obstruction shows up that cannot be ``seen'' without real $K$-theory.

The present author observed that there is a much more extensive
obstruction theory when $M$ is not simply connected.  Take the
fundamental group $\pi$ of $M$, a countable discrete group.
Complete the real group ring $\bR\pi$ in its greatest $C^*$-norm to get the
real group {\Ca} $A=C^*_{\bR}(\pi)$. (Alternatively, one could use
the reduced real group {\Ca} $A_r=C^*_{\bR, r}(\pi)$, the completion of
the group ring for its left action on $L^2(\pi)$.  For present purposes
it doesn't much matter.) Coupling the Dirac operator on $M$
to the universal flat $C^*_{\bR}(\pi)$-bundle $\widetilde M \times_\pi
A$ over $M$, one gets a Dirac index with values in $KO_n(A)$, which
must vanish if $M$ has {\psc}.  Thus we have a new source of
obstructions to {\psc}.

As shown in \cite{MR842428,MR1133900}, this $KO_n(A)$-valued index
obstruction can be computed to be $\mu\circ f_*(\alpha_M)$, where
$\alpha_M\in KO_n(M)$ is the ``Atiyah orientation'' of $M$, i.e., the
$KO$-fundamental class defined by the spin structure, $f\co M\to B\pi$
is a classifying map for the universal cover $\widetilde M \to M$,
and $\mu\co KO_n(B\pi) \to KO_n(A)$ is the ``real assembly map,''
closely related to the Baum-Connes assembly map in
\cite{MR1292018}.\footnote{The relationship is this. Let
  $\underline{E}\pi$ denote the universal proper $\pi$-space and let
  $E\pi$ denote the universal free $\pi$-space. These coincide if and
  only if $\pi$ is torsion-free.  The Baum-Connes assembly map is
  defined on $KO_\bullet^\pi(\underline{E}\pi)$ whereas our map is
  defined on $KO_\bullet^\pi({E}\pi) = KO_n(B\pi)$. Since $E\pi$ is a
  proper $\pi$-space, there is a canonical $\pi$-map $E\pi \to
  \underline{E}\pi$ (unique up to equivariant homotopy)
  and so our $\mu$ factors through the Baum-Connes
  assembly map, and agrees with it if $\pi$ has no torsion.}

In \cite[Theorem 2.5]{MR1133900}, I showed that for $\pi$ finite, the
image of the reduced assembly map $\mu$ (what one gets after pulling out the
contribution from the trivial group, i.e., the Lichnerowicz and
Hitchin obstructions) is precisely the image in $KO_\bullet(\bR\pi)$
of the $2$-torsion in $KO_\bullet(\bR\pi_2)$, $\pi_2\subseteq \pi$ a Sylow
$2$-subgroup. This lives in degrees $1$ and $2$ mod $4$ and comes from the
irreducible representations of $\pi_2$ of real and quaternionic type,
in the sense that we explained in Theorem \ref{thm:AandAC}.  So far
the obstructions detected by $\mu$ are the only known obstructions to
{\psc} on closed spin manifolds of dimension $>4$ with finite
fundamental group.

The problem of existence or non-existence of {\psc} on a spin manifold
can be split into two parts, one ``stable'' and one ``nonstable.''
Stability here refers to taking the product with enough copies of a
``Bott manifold'' $\mathrm{Bt}^8$, a simply connected closed Ricci-flat
$8$-manifold 
representing the generator of Bott periodicity.  Since index
obstructions in $K$-theory of real {\Ca}s live in groups which are
periodic mod $8$, stabilizing the problem by crossing with copies of
$\mathrm{Bt}^8$ compensates for this by introducing
$8$-periodicity on the geometric
side.  Indeed it was shown in \cite{MR1321004} that the stable
conjecture ($M\times \mathrm{Bt}^k$ admits a metric of {\psc} for
sufficiently large $k$ if and only if $\mu\circ f_*(\alpha_M)$
vanishes) holds when $\pi$ is finite.  Stolz has extended this theorem
as follows: for a completely general closed spin manifold $M^n$ with
fundamental group $\pi$, $M\times \mathrm{Bt}^k$ admits a metric of {\psc} for
sufficiently large $k$ if and only if $\mu\circ f_*(\alpha_M)$
vanishes, provided that the real Baum-Connes conjecture (bijectivity
of the Baum-Connes assembly map $\mu\co KO_\bullet^\pi({E}\pi)\to
KO_\bullet(C^*_{\bR,r})$) holds for $\pi$.  In fact, Baum-Connes can be
weakened here in two ways --- one only needs injectivity of $\mu$, not
surjectivity, and one can replace $C^*_{\bR,r}$ by the full {\Ca}
$C^*_{\bR}$.  Since the full {\Ca} surjects onto the reduced {\Ca},
injectivity of the assembly map for the full {\Ca} is a weaker condition.
Sketches of Stolz's theorem may be
found in \cite{MR1403963,MR1937026}, though unfortunately the full
proof of this was never written up.

Since the real version of the Baum-Connes conjecture has just been
seen to play a fundamental role here, it is worth remarking that 
the real and complex versions of the Baum-Connes conjecture are
actually equivalent \cite{MR2082090,MR2077669}.  Thus the real
Baum-Connes conjecture holds in the huge number of cases where the
complex Baum-Connes conjecture has been verified.

\subsection{Representation theory}
\label{sec:rep}

Since we have already mentioned the real Baum-Connes conjecture, it is
worth mentioning that this, as well as the general theory of real {\Ca}s,
has some relevance to representation theory.  Suppose $G$ is a locally
compact group (separable, say, but not necessarily discrete).  The
real group {\Ca} $C^*_{\bR}(G)$ 
is the completion of the real $L^1$-algebra (the convolution algebra
of real-valued $L^1$ functions on $G$) for the maximal {\Ca}
norm. Obviously this defines a canonical real structure on the complex
group {\Ca} $C^*(G)$, and similarly we have $C^*_{\bR,r}(G)$ inside
the reduced {\Ca} $C^*_r(G)$.  Computing the structure of
$C^*_{\bR}(G)$ or of $C^*_{\bR,r}(G)$ gives us more information than
just computing the structure of their complexifications.  For
instance, it gives us the type classification of the representations,
as we saw in Theorem \ref{thm:AandAC} and Example \ref{ex:types}.
The real Baum-Connes conjecture, when it's known to hold, gives us at
least partial information on the structure of $C^*_{\bR,r}(G)$ (its
$K$-theory).

Here are some simple examples (where the real structure is not totally
uninteresting) to illustrate these ideas.
\begin{example}
\label{ex:SU2}
Let $G=SU(2)$. As is well known, this has (up to equivalence)
irreducible complex representation of each positive integer
dimension. It is customary to parameterize the representations $V_k$ by the
value of the ``spin'' $k=0, \frac12, 1, \frac32, \cdots$ (this is the
highest weight divided by the unique positive root), so that
$\dim V_k=2k+1$. The character $\chi_k$ of $V_k$ is given on a
maximal torus by $e^{i\theta}\mapsto \frac{\sin
  (2k+1)\theta}{\sin\theta}$, which is real-valued, and thus all the
representations must have real or quaternionic type. In fact, $V_k$ is
of real type if $k$ is an integer and is of quaternionic type if $k$
is a half-integer.  (That's because $V_1$ is the complexification of
the covering map $SU(2)\to SO(3)$, while $V_{1/2}$ acts on the unit
quaternions, and all the other representations can be obtained from
these by taking 
tensor products and decomposing. A tensor product of real
representations is real, and a tensor product of a real representation
with a quaternionic one is quaternionic.) Indeed if one computes
the Frobenius-Schur indicator
$\int_G\chi_k(g^2)\,dg$ for $V_k$ using the Weyl integration formula,
one gets
\[
\begin{aligned}
  \int_G\chi_k&(g^2)\,dg = \frac{1}{\pi}\int_0^{2\pi} \Bigl(e^{4ki\theta}
  +e^{4(k-1)i\theta}+\cdots+ e^{-4ki\theta}
  \Bigr)\,{\sin^2\theta}\,d\theta \\
  &= \frac{1}{4\pi}\int_0^{2\pi} \Bigl(e^{4ki\theta}
  +e^{4(k-1)i\theta}+\cdots+ e^{-4ki\theta}
  \Bigr)\,\Bigl(2 - e^{2i\theta}-   e^{-2i\theta} \Bigr)\,d\theta .
\end{aligned}
\]
If $k$ is an integer, we get
\[
\frac{1}{4\pi}\int_0^{2\pi}
\Bigl((\cdots + e^{4i\theta})+1+(e^{-4i\theta}+\cdots) \Bigr)\,
\Bigl(2 - e^{2i\theta}-   e^{-2i\theta} \Bigr)\,d\theta = 1,
\]
where the terms in small parentheses are missing if $k=0$, while if $k$ is a
half-integer, we get 
\[
\frac{1}{4\pi}\int_0^{2\pi}
\Bigl(\cdots + e^{2i\theta}+e^{-2i\theta}+\cdots \Bigr)\,
\Bigl(2 - e^{2i\theta}-   e^{-2i\theta} \Bigr)\,d\theta = -1.
\]
This confirms the type classification we gave earlier.

Thus $C^*_\bR(G) \cong \bigoplus_{k\in \bN} M_{2k+1}(\bR) \oplus
\bigoplus_{k=\frac12,\frac32, \cdots} M_{k+\frac12}(\bH)$. In
particular, we see that $KO_\bullet(C^*_\bR(G))\cong
(KO_\bullet)^\infty \oplus (KSp_\bullet)^\infty$, and in degrees $1$
and $2$ mod $4$, this is an infinite direct sum of copies of $\bZ/2$,
whereas the torsion-free contributions appear only in degrees
divisible by $4$.  Conversely, if one had some independent method of
computing $KO_\bullet(C^*_\bR(G))$, it would immediately tell us that
$G$ has no irreducible representations of complex type, and infinitely
many representations of both real and of quaternionic type.
\end{example}
\begin{example}
\label{ex:Weil}
Let $H$ be the compact group $\bT \cup j\,\bT$, where $j$ is an
element with $j^2=-1$, $jzj^{-1}=\overline{z}$ for $z\in \bT$. This is
a nonsplit extension of $\text{Gal}(\bC/\bR)\cong
\bZ/2$ by $\bT$ and is secretly the maximal
compact subgroup of $W_\bR$, the Weil group of the reals (which splits
as $\bR^\times_+\times H$). The induced action of $j$ on
$\widehat{\bT}=\bZ$ sends $n\mapsto -n$. So the Mackey machine tells
us that the irreducible complex representations of $H$ are the
following:
\begin{enumerate}
\item two one-dimensional representations $\chi_0^\pm$ which are
  trivial on $\bT$ and send $j\mapsto \pm 1$. These representations
  are obviously of real type.
\item a family $\pi_n=\Ind_{\bT}^H \sigma_n, \, n\in
  \bZ\smallsetminus \{0\}$ of two-dimensional representations, where
  $\sigma_n(z) = z^n,\,z\in \bT$. These representations are all of
  quaternionic type since they come from complexifying the
  representation $H\to \bH^\times$ given by $z\mapsto z^n$, $j\mapsto j$.
\end{enumerate}
We immediately conclude that $C^*_{\bR}(H)\cong \bR\oplus \bR \oplus
(\bH)^\infty$. Thus\break $KO_\bullet(C^*_{\bR}(H))$ is elementary abelian
of rank $2$ in degrees $1$ and $2$ mod $8$, and is $(\bZ/2)^\infty$ in
degrees $5$ and $6$ mod $8$.  Again, if we had an independent way to
compute $KO_\bullet(C^*_{\bR}(H))$, it would tell us about the types
of the representations.
\end{example}
\begin{example}
\label{ex:SL2}
A slightly more interesting example is $G=SL(2,\bC)$, a simple complex
Lie group with $K=SU(2)$ as maximal compact subgroup. The reduced dual
of $G$ is Hausdorff, and the complex reduced {\Ca} $C^*_r(G)$ is a
stable continuous-trace algebra with trivial Dixmier-Douady invariant,
i.e., it is Morita equivalent to $C_0(\widehat G_r)$.  All the
irreducible complex representations of $C^*_r(G)$ are principal series
representations, and all unitary principal series are irreducible.
Thus we see that $C^*_r(G)$ is Morita equivalent to $C_0(\widehat
M/W)$, where $M$ is a Cartan subgroup, which we can take to be
$\bC^\times$, and $W=\{\pm1\}$ is the Weyl group, which acts on
$\widehat{\bC^\times} \cong \bZ\times \bR$ by $-1$ (on both factors).
So $C^*_r(G)$ is Morita equivalent to $C_0([0,\infty)) \oplus
  \bigoplus_{n\ge 1} C_0(\bR)$, with $[0,\infty)= (\{0\}\times \bR)/W$.
(For all of this one can see \cite[V]{MR0164248} or \cite{MR724030},
for example.) The complex Baum-Connes map gives an isomorphism
$K_\bullet^G(G/K)\cong R(K)\otimes K_{\bullet-3} \xrightarrow{\mu}
K_\bullet(C^*_r(G))$, sending the generator of the representation ring
$R(K)$ associated to $V_k$ (in the notation
of Example \ref{ex:SU2}) to the generator of the $\bZ$ summand in
$K_0(C^*_r(G))$ associated to the principal series with discrete parameter
$\pm (2k+1)$.  There is no contribution to $K_0(C^*_r(G))$ from the 
spherical principal series (corresponding to the
fixed point $n=0$ of $W$ on $\bZ$) since $\bR/\{\pm1\}\cong [0, \infty)$ is
properly contractible. 

Now let's analyze the real structure.  We can start by looking at
$K$-types.  The group $SL(2,\bC)$ is the double cover of the Lorentz
group $SO(3,1)_0$.  Representations that descend to $SO(3,1)_0$ must
have $K$ types that factor through $SU(2)\to SO(3)$, and so have
integral spin.  All integral spin representations have real type, so
these representations are also of real type, at least when restricted
to $K$. The genuine representations of $SL(2,\bC)$ that do not descend
to $SO(3,1)_0$ must have $K$-types with half-integral spin, and these
representations are of quaternionic type, at least when restricted
to $K$.  There is one principal series which is obviously of real
type, namely the ``$0$-point'' of the spherical principal series,
since this representation is simply $\Ind_B^G 1$, where $B$ is a Borel
subgroup.  Since the trivial representation of $B$ is of real type, we
get a real form for the complex induced representation by using
induction with real 
Hilbert spaces instead.  And thus we get an irreducible real
representation on $L^2_{\bR}(G/B)\cong L^2_{\bR}(K/\bT)$.  In fact the
other spherical principal series can be realized on this same Hilbert
space (see \cite[p.\ 261]{MR0164248}) so they, too, are of real type.
But this method won't work for other characters of
$B$ since none of the other one-dimensional unitary characters of $\bC^\times$
are of real type.  If we look at a principal series representation of
$G$ with discrete parameter $\pm n\in \bZ$, its restriction to $K$ can
be identified with $\Ind_{\bT}^K\chi_n$, which by Frobenius
reciprocity contains $V_k$ with multiplicity equal to the multiplicity
of $\chi_n$ in $V_k$. This is $0$ if $2k$ and $n$ have opposite parity
or if $|n|>2k$, and is $1$ otherwise.  So this principal series
representation has all its $K$-types of multiplicity $1$ and has real
(resp., quaternionic) type when restricted to $K$ provided $n$ is even
(odd).

We can analyze things in more detail by seeing what the involution
(of Proposition \ref{prop:Ahatinv}) on
$\widehat G$ does to the principal series. Clearly it sends
$\Ind_B^G\chi$ to $\Ind_B^G\overline\chi$, if $\chi$ is a
one-dimensional representation of $M$ viewed as a representation of
$B$. But since $W=\{\pm 1\}$, $\overline{\chi} = w\cdot \chi$, for $w$
the generator of $W$, and we get an equivalent representation.  Thus
\emph{the involution on $\widehat{G}_r$ is trivial}. One can also
check this very easily by observing that all the characters of $G$ are
real-valued. (See for example \cite[Theorem 5.5.3.1]{MR0498999}, 
where again the key fact for us is that $w\cdot \chi = \overline{\chi}$.)
J.\ Adams has studied this property in much greater
generality and proved:
\begin{theorem}[Adams {\cite[Theorem 1.8]{MR3292297}}]
\label{thm:Adams}
If $G$ is a connected reductive algebraic group over $\bR$ with
maximal compact subgroup $K$, if $-1$ lies in the Weyl group of the
complexification of $G$, and if every irreducible representation
of $K$ is of real or quaternionic type, then every unitary
representation of $G$ is also of real or quaternionic type.
\end{theorem}
  
Thus we know that the involution on $\widehat G_r$ is trivial and that
$C^*_r(G)$ is a real {\Ca} of continuous trace, with spectrum a
countable union of contractible components, all but one of which are
homeomorphic to $\bR$, with the exceptional component homeomorphic to
$[0,\infty)$. The real Dixmier-Douady invariants must all vanish since
$H^2(\widehat G_r, \bZ/2)=0$.  The sign choice invariants are now
determined by the $K$-types, since all the $K$-types have multiplicity
one and thus a representation of $G$ of real (resp., quaternionic)
type must have all its $K$-types of the same type.
Putting everything together, we see that we have proved the following
theorem:
\begin{theorem}
\label{thm:SL2C}
The reduced real {\Ca} of $SL(2,\bC)$ is a stable real
continuous-trace algebra, Morita equivalent to
\[
C_0^{\bR}([0,\infty)) \oplus \bigoplus_{n>0\text{ even}}
  C_0^{\bR}(\bR) \oplus \bigoplus_{n>0\text{ odd}} C_0^{\bH}(\bR).
\]
\end{theorem}

Schick's proof in \cite{MR2077669} that the complex Baum-Connes
conjecture implies the real Baum-Connes conjecture is stated only for
discrete groups, but it goes over without difficulty to general
locally compact groups, at least in the case of trivial coefficients.
Since the complex Baum-Connes conjecture (without coefficients) is
known for all connective reductive Lie groups \cite{MR894996,MR1914617}, 
the real Baum-Connes map
is an isomorphism $KO_\bullet^G(G/K) \xrightarrow{\mu} KO_\bullet(C^*_{\bR,r}(G))$.
This by itself gives some information on the real
structure of the various summands in $C^*_{\bR,r}(G)$. Since $G/K$ has
a $G$-invariant spin structure in this case, by the results of
\cite[\S5]{MR918241}, $KO_\bullet^G(G/K) \cong KO_\bullet^K(G/K) \cong
KO_\bullet^K(\fp)$, which by equivariant Bott periodicity is
$KO_{\bullet-3}(C^*_{\bR}(K))$, which we computed in Example
\ref{ex:SU2}. (Here $\fp$ is the orthogonal complement to the Lie
algebra of $K$ inside the Lie algebra of $G$. In this case, $\fp$ is
isomorphic as a $K$-module to the adjoint representation of $K$.)
On the other side of the isomorphism, we have
$KO_\bullet(C_0^{\bR}(\bR)) \cong KO^{-\bullet}(\bR)\cong
KO^{-\bullet-1}(\pt) \cong KO_{\bullet+1}$, and similarly
$KO_\bullet(C_0^{\bH}(\bR)) \cong KSp^{-\bullet}(\bR)\cong
KO^{-\bullet-5}(\pt) \cong KO_{\bullet+5}$. So the torsion-free
summands in $KO_\bullet(C^*_{\bR,r}(G))$ are all in degrees $3$ mod $4$.
Since there are no torsion-free summands in $KO_\bullet(C^*_{\bR,r}(G))$
in degrees $1$
mod $4$, we immediately conclude that no unitary principal series
(except perhaps for the spherical principal series, which don't
contribute to the $K$-theory) are of complex type, and that there are
infinitely many lines of principal series
of both real and quaternionic type.  This is
a large part of Theorem \ref{thm:SL2C}, and is not totally
trivial to check directly. 

One interesting feature of the Baum-Connes isomorphism is the degree
shift.  Since $KO_\bullet^G(G/K) \cong KO_{\bullet-3}(C^*_{\bR}(K))$,
while $KO_\bullet(C^*_{\bR,r}(G))$ is a sum of copies of
$KO_{\bullet+1}$, associated to principal series with $K$-types of
integral spin and $KO_{\bullet+5}$, associated to principal series
with $K$-types of half-integral spin, we see that representations of
$K$ of integral spin on the left match with those of half-integral
spin on the right, and those of half-integral spin on the left match
with those of integral spin on the right.  This is due to the
``$\rho$-shift'' in Dirac induction. If we were to replace $SL(2,\bC)$
by the adjoint group $G=PSL(2,\bC)$, the maximal compact subgroup would
become $K=SO(3)$, with $K$-types only of integral spin, but on the left,
since $G/K$ would no longer have a $G$-invariant spin structure, 
$KO_\bullet(G/K)$ would be given by genuine
representations of the double cover of $K$, i.e., representations only
of half-integral spin.
\end{example}
\begin{example}
\label{ex:classicalgps}
The techniques we used in Example \ref{ex:SL2} and Theorem
\ref{thm:SL2C} can also be used to compute the reduced real
{\Ca}s of arbitrary connected complex reductive Lie groups. A useful
starting point is \cite[Proposition 4.1]{MR724030}, which states that
for such a group $G$, $\widehat G_r$ is Hausdorff and
$C^*_r(G)\cong C_0(\widehat G_r)\otimes \cK$ is a stable
continuous-trace algebra with trivial Dixmier-Douady class.
The cases of $SO(4n,\bC)$, $SO(2n+1,\bC)$, and $Sp(n,\bC)$ are
particularly easy (and interesting). By \cite[Theorem
7.7]{MR0252560} or \cite[VI.(5.4)(vi)]{MR1410059}, all representations of
$SO(2n+1)$ are of real type, and by \cite[Theorem
7.9]{MR0252560} or \cite[VI.(5.5)(ix)]{MR1410059}, all representations of
$SO(4n)$ are of real type.  Thus we obtain
\begin{theorem}
\label{thm:SOnC}
Let $G=SO(2n+1,\bC)$ or $SO(4n,\bC)$.  Then
\[
C^*_{r,\bR}(G)\cong C_0^{\bR}(\widehat G_r)\otimes \cK_{\bR}
\]
is a stable real
continuous-trace algebra with trivial Dixmier-Douady class.
\end{theorem}
\begin{proof}
By Theorem \ref{thm:Adams}, the involution on $\widehat G_r$ is
trivial, and the sign choice invariant has to be constant on each
connected component.  Let $M$ be a Cartan subgroup of $G$, let $W$ be
its Weyl group, let $B=MN$
be a Borel subgroup, let $K$
be a maximal compact subgroup (an orthogonal group $SO(2n+1)$ or
$SO(4n)$), and let $T=K\cap M$, a maximal torus in $K$.
The irreducible representations in the reduced dual are all
principal series $\Ind_B^G \chi$, where $\chi$ is a character of $M$
extended to a character of $B$ by taking it to be trivial on $N$.
When restricted to $K$, this is the same as $\Ind_T^K\chi|_T$.
If $\rho\in \widehat K$ is a $K$-type, then by Frobenius reciprocity,
$\rho$ appears in this induced representation with multiplicity equal
to the multiplicity of $\chi|_T$ in $\rho|_T$.  So given $\chi$, take
$\rho$ to have highest weight in the $W$-orbit of $\chi|_T$, and we see
that $\rho$ occurs in $\Ind_B^G \chi$ with multiplicity $1$.  Since
$\rho$ is real and $\Ind_B^G \chi$ is of either real or quaternionic
type, we see that its being of real type is the only possibility.
(Otherwise the invariant skew-symmetric form on the representation
would restrict to a skew-symmetric invariant form on $\rho$.)
Thus $C^*_{r,\bR}(G)$ is a stable real continuous-trace algebra with all
irreducible representations of real type.  The Dixmier-Douady
invariant has to vanish since as pointed out in
\cite[p.\ 277]{MR724030}, all connected components of $\widehat G_r$
are contractible.
\end{proof}
In a similar fashion we have
\begin{theorem}
\label{thm:SpnC}
Let $G=Sp(n,\bC)$, let $M$ be a Cartan subgroup, and let $W$ be its
Weyl group. Let $K=Sp(n)$, a
maximal compact subgroup, and let $T=M\cap K$, a maximal torus in $K$
Then $C^*_{r,\bR}(G)$ is a stable real
continuous-trace algebra which is Morita equivalent to a direct sum of
pieces of the form $C_0^\bR(Y)$ and $C_0^\bH(Y)$.  Here $Y$ ranges
over the components of $\widehat M/W$.  Infinitely many pieces of each
type {\lp}real or quaternionic{\rp} occur. If $\chi\in \widehat M$ and
$\chi|_T$ is its ``discrete parameter'', then the associated summand
in $C^*_{r,\bR}(G)$ is of real type if and only if the representation
of $K$ with highest weight in the $W$-orbit of $\chi$ is of real type.
\end{theorem}
\begin{proof}
This is exactly the same as the proof of Theorem \ref{thm:SOnC}, the only
difference being that ``half'' of the representations of $K$ are of
quaternionic type (see \cite[Theorem 7.6]{MR0252560} and
\cite[VI.(5.3)(vi)]{MR1410059}). 
\end{proof}
\end{example}  

\subsection{Orientifold string theories}
\label{sec:strings}

A last area where real {\Ca}s and their $K$-theory seem to play a
significant role is in the study of orientifold string theories in
physics.  (See for example \cite[\S8.8]{MR2151029} and
\cite[Ch.\ 13]{MR2151030}.)
Such a theory is based on a spacetime manifold $X$ equipped
with an involution $\iota$, and the theory is based on a ``sigma
model'' where the fundamental ``strings'' are equivariant maps from a
``string worldsheet'' $\Sigma$ (an oriented Riemann surface) to $X$,
equivariant with respect to the ``worldsheet
parity operator'' $\Omega$, an orientation-reversing involution on
$\Sigma$, and $\iota$, the involution on $X$.  Restricting attention only
to equivariant strings is basically what physicists often call
\emph{GSO projection}, after Gliozzi-Scherk-Olive \cite{GSO}, and introduces
enough flexibility in the theory to get rid of lots of unwanted
states. In order to preserve a reasonable amount of supersymmetry,
usually one assumes that the spacetime manifold (except for a flat
Minkowski space factor, which we can ignore here) is chosen to be a
Calabi-Yau manifold, that is, a complex K\"ahler manifold $X$ with
vanishing first Chern class, and that the involution of $X$ is either
holomorphic or anti-holomorphic.
If we choose $X$ to be compact, then in
low dimensions there are very few possibilities --- if $X$ has complex
dimension $1$, then it is an elliptic curve, and if $X$ has complex
dimension $2$, then it is either a complex torus or a $K3$ surface.
In the papers \cite{MR3267662,MR3316647}, the case of an elliptic
curve was treated in great detail. Example \ref{ex:K3} and
Theorem \ref{thm:DDK3} were motivated by the case of $K3$ surfaces,
which should also be of great physical interest.

Now we need to explain the connection with real {\Ca}s and their
$K$-theory.  An orientifold string theory comes with two kinds of
important submanifolds of the spacetime manifold $X$: \emph{D-branes},
which are submanifolds on which ``open'' strings --- really, compact
strings with boundary --- can begin or end
(where we specify boundary conditions of Dirichlet or Neumann type),
and \emph{O-planes}, which are the connected components of the fixed
set of the involution $\iota$.  There are charges attached to these
two kinds of submanifolds. D-branes have charges in $K$-theory
\cite{MR1606278,MR1674715}, where
the kind of $K$-theory to be used depends on the specific details of
the string theory, and should be a variant of $KR$-theory for
orientifold theories.  The O-planes have $\pm$ signs which determine
whether the Chan-Paton bundles restricted to them have real or
symplectic type.  These sign choices result in ``twisting'' of the
$KR$-theory, such as appeared above in Definition \ref{def:sign}. In
addition, there is a further twisting of the $KR$-theory due to the
``$B$-field'' which appears in the Wess-Zumino term in the string
action.  It would be too complicated to explain the physics involved,
but mathematically, the $B$-field gives rise to a class in Moutuou's
$H^2_\iota(X, \cT)$. But in short, the effect of the O-plane charges
and the $B$-field is to make the D-brane charges live in twisted
$KR$-theory, i.e., in the $K$-theory of a real continuous-trace
algebra determined by the O-plane charges and the $B$-field.  In this
way, (type II) orientifold string theories naturally lead
to $K$-theory of real continuous-trace algebras, which is
most interesting from the point of view of physics when $(X,\iota)$ is
a Calabi-Yau manifold with a holomorphic or anti-holomorphic
involution \cite{MR3267662,MR3316647}.

An important aspect of string theories is the existence of
\emph{dualities} between one theory and another.  These are cases
where two seemingly different theories predict the same observable
physics, or in other words, are equivalent descriptions of the same
physical system.  The most important examples of such dualities are
\emph{T-duality}, or target-space duality, where the target space $X$
of the model is changed by replacing tori by their duals, and the very
closely related \emph{mirror symmetry} of Calabi-Yau manifolds. These
dualities do not have to preserve the type of the theory (IIA or IIB)
--- in fact, in the case of T-duality in a single circle, the type is
reversed --- and they frequently change the geometry or topology of
the spacetime and/or the twisting (sign choice and/or
Dixmier-Douady class).  The possible theories with $X$ an elliptic
curve and $\iota$ holomorphic or anti-holomorphic were studied in
\cite{Gao:2010ava,MR3267662,MR3316647}, and found to be grouped into 3
classes, each containing 3 or 4 different theories.  All of the
theories in a single group are related to one another by dualities,
and theories in two different groups can never be related by
dualities.  One way to see this is via the twisted $KR$-theory
classifying the D-brane charges.  Theories which are dual to one
another must have twisted $KR$-groups which agree up to a degree
shift, while if the $KR$-groups are non-isomorphic (even after a
degree shift), the two theories cannot possibly be equivalent.
Thus calculations of twisted $KR$-theory provide a methodology for
testing conjectures about dualities in string theory.

In the case where $X$ is an elliptic curve and $\iota$ is holomorphic
or anti-holomorphic, the twisted $KR$-groups were computed in
\cite{MR3267662,MR3316647}. In one group of three theories ($\iota$
the identity, $\iota$ anti-holomorphic with a fixed set with two
components and with trivial sign choice, and $\iota$ holomorphic with 
four isolated fixed points), the $KR$-theory turned out to be
$KO^{-\bullet}(T^2)$, up to a degree shift. In the next group ($\iota$
holomorphic and free, $\iota$ anti-holomorphic and free, $\iota$
holomorphic with four fixed points with sign choice $(+,+,-,-)$, and
$\iota$ anti-holomorphic with a fixed set with two
components and sign choice $(+,-)$), the
groups turned out to be $KSC^{-\bullet}\oplus KSC^{-\bullet-1}$ up to
a degree shift.  In the last group ($\iota$ the identity but the
Dixmier-Douady invariant ($B$-field) nontrivial, $\iota$ holomorphic with 
four isolated fixed points and sign choice $(+,+,+,-)$, and $\iota$
anti-holomorphic with fixed set a circle), the $KR$-theory turned out
to be $KO^{-\bullet}\oplus KO^{-\bullet}\oplus K^{-\bullet-1}$ up to a
degree shift.  The $KR$-groups in one T-duality group are not
isomorphic to those in another, so there cannot be any additional
dualities between theories.

Rather curiously, it turns out (as was shown in \cite{MR3305978}) that
all of the isomorphisms of twisted $KR$-groups associated to dualities
of elliptic curve orientifold theories arise from the real Baum-Connes
isomorphisms for certain solvable groups with $\bZ$ or $\bZ^2$ as a
subgroup of finite index.  This suggests a rather mysterious
connection between representation theory and duality for string
theories, which we intend to explore further.  It will be especially
interesting to study dualities between orientifold theories
compactified on abelian varieties of dimension $2$ or $3$ and on
K3-surfaces, and ultimately on simply connected Calabi-Yau $3$-folds.

\bibliographystyle{amsplain} \bibliography{realC} 

\def\cprime{$'$}
\providecommand{\bysame}{\leavevmode\hbox to3em{\hrulefill}\thinspace}
\providecommand{\MR}{\relax\ifhmode\unskip\space\fi MR }
\providecommand{\MRhref}[2]{%
  \href{http://www.ams.org/mathscinet-getitem?mr=#1}{#2}
}
\providecommand{\href}[2]{#2}
\begin{thebibliography}{10}

\bibitem{MR0252560}
J.~Frank Adams, \emph{Lectures on {L}ie groups}, W. A. Benjamin, Inc., New
  York-Amsterdam, 1969, Reprinted by Univ. of Chicago Press, 1982. \MR{0252560
  (40 \#5780)}

\bibitem{MR3292297}
Jeffrey Adams, \emph{The real {C}hevalley involution}, Compos. Math.
  \textbf{150} (2014), no.~12, 2127--2142. \MR{3292297}

\bibitem{Anderson}
D.~W. Anderson, \emph{The real {$K$}-theory of classifying spaces}, Proc. Nat.
  Acad. Sci. U. S. A. \textbf{51} (1964), no.~4, 634--636.

\bibitem{MR0025453}
Richard~F. Arens and Irving Kaplansky, \emph{Topological representation of
  algebras}, Trans. Amer. Math. Soc. \textbf{63} (1948), 457--481. \MR{0025453
  (10,7c)}

\bibitem{MR0206940}
M.~F. Atiyah, \emph{{$K$}-theory and reality}, Quart. J. Math. Oxford Ser. (2)
  \textbf{17} (1966), 367--386. \MR{0206940 (34 \#6756)}

\bibitem{MR2172633}
Michael Atiyah and Graeme Segal, \emph{Twisted {$K$}-theory}, Ukr. Mat. Visn.
  \textbf{1} (2004), no.~3, 287--330. \MR{2172633 (2006m:55017)}

\bibitem{MR1292018}
Paul Baum, Alain Connes, and Nigel Higson, \emph{Classifying space for proper
  actions and {$K$}-theory of group {$C^\ast$}-algebras}, {$C^\ast$}-algebras:
  1943--1993 ({S}an {A}ntonio, {TX}, 1993), Contemp. Math., vol. 167, Amer.
  Math. Soc., Providence, RI, 1994, pp.~240--291. \MR{1292018 (96c:46070)}

\bibitem{MR2082090}
Paul Baum and Max Karoubi, \emph{On the {B}aum-{C}onnes conjecture in the real
  case}, Q. J. Math. \textbf{55} (2004), no.~3, 231--235. \MR{2082090
  (2005d:19005)}

\bibitem{MR2233367}
Victor~M. Bogdan, \emph{On {F}robenius, {M}azur, and {G}elfand-{M}azur theorems
  on division algebras}, Quaest. Math. \textbf{29} (2006), no.~2, 171--209.
  \MR{2233367 (2007c:46044)}

\bibitem{MR1410059}
Theodor Br{\"o}cker and Tammo tom Dieck, \emph{Representations of compact {L}ie
  groups}, Graduate Texts in Mathematics, vol.~98, Springer-Verlag, New York,
  1995, Translated from the German manuscript, Corrected reprint of the 1985
  translation. \MR{1410059 (97i:22005)}

\bibitem{MR0435864}
A.~Connes, \emph{A factor not anti-isomorphic to itself}, Bull. London Math.
  Soc. \textbf{7} (1975), 171--174. \MR{0435864 (55 \#8815)}

\bibitem{MR0370209}
\bysame, \emph{A factor not anti-isomorphic to itself}, Ann. Math. (2)
  \textbf{101} (1975), 536--554. \MR{0370209 (51 \#6438)}

\bibitem{MR0458185}
Jacques Dixmier, \emph{{$C\sp*$}-algebras}, North-Holland Publishing Co.,
  Amsterdam-New York-Oxford, 1977, Translated from the French by Francis
  Jellett, North-Holland Mathematical Library, Vol. 15. \MR{0458185 (56
  \#16388)}

\bibitem{MR3267662}
Charles Doran, Stefan M{\'e}ndez-Diez, and Jonathan Rosenberg, \emph{T-duality
  for orientifolds and twisted {KR}-theory}, Lett. Math. Phys. \textbf{104}
  (2014), no.~11, 1333--1364. \MR{3267662}

\bibitem{MR3316647}
\bysame, \emph{String theory on elliptic curve orientifolds and {$KR$}-theory},
  Comm. Math. Phys. \textbf{335} (2015), no.~2, 955--1001. \MR{3316647}

\bibitem{MR0254839}
Johan~L. Dupont, \emph{Symplectic bundles and {$KR$}-theory}, Math. Scand.
  \textbf{24} (1969), 27--30. \MR{0254839 (40 \#8046)}

\bibitem{MR0397420}
George~A. Elliott, \emph{On the classification of inductive limits of sequences
  of semisimple finite-dimensional algebras}, J. Algebra \textbf{38} (1976),
  no.~1, 29--44. \MR{0397420 (53 \#1279)}

\bibitem{MR0164248}
J.~M.~G. Fell, \emph{The structure of algebras of operator fields}, Acta Math.
  \textbf{106} (1961), 233--280. \MR{0164248 (29 \#1547)}

\bibitem{Gao:2010ava}
Dongfeng Gao and Kentaro Hori, \emph{{On the structure of the {C}han-{P}aton
  factors for {D}-branes in type {II} orientifolds}}, preprint,
  arXiv:1004.3972, 2010.

\bibitem{MR728909}
T.~Giordano, \emph{Antiautomorphismes involutifs des facteurs de von {N}eumann
  injectifs. {I}}, J. Operator Theory \textbf{10} (1983), no.~2, 251--287.
  \MR{728909 (85h:46087)}

\bibitem{MR931219}
\bysame, \emph{A classification of approximately finite real {$C^*$}-algebras},
  J. Reine Angew. Math. \textbf{385} (1988), 161--194. \MR{931219 (89h:46078)}

\bibitem{MR564145}
Thierry Giordano and Vaughan Jones, \emph{Antiautomorphismes involutifs du
  facteur hyperfini de type {${\rm II}_{1}$}}, C. R. Acad. Sci. Paris S\'er.
  A-B \textbf{290} (1980), no.~1, A29--A31. \MR{564145 (82b:46079)}

\bibitem{GSO}
F.~Gliozzi, Joel Scherk, and David~I. Olive, \emph{Supersymmetry, supergravity
  theories and the dual spinor model}, Nuclear Phys. B \textbf{122} (1977),
  no.~2, 253--290.

\bibitem{MR0345092}
Roger Godement, \emph{Topologie alg\'ebrique et th\'eorie des faisceaux},
  Hermann, Paris, 1973, Troisi{\`e}me {\'e}dition revue et corrig{\'e}e,
  Publications de l'Institut de Math{\'e}matique de l'Universit{\'e} de
  Strasbourg, XIII, Actualit{\'e}s Scientifiques et Industrielles, No. 1252.
  \MR{0345092 (49 \#9831)}

\bibitem{MR677280}
K.~R. Goodearl, \emph{Notes on real and complex {$C^{\ast} $}-algebras}, Shiva
  Mathematics Series, vol.~5, Shiva Publishing Ltd., Nantwich, 1982. \MR{677280
  (85d:46079)}

\bibitem{MR904013}
K.~R. Goodearl and D.~E. Handelman, \emph{Classification of ring and
  {$C^*$}-algebra direct limits of finite-dimensional semisimple real
  algebras}, Mem. Amer. Math. Soc. \textbf{69} (1987), no.~372, viii+147.
  \MR{904013 (88k:46067)}

\bibitem{MR0164347}
Paul~S. Green, \emph{A cohomology theory based upon self-conjugacies of complex
  vector bundles}, Bull. Amer. Math. Soc. \textbf{70} (1964), 522--524.
  \MR{0164347 (29 \#1644)}

\bibitem{MR0102537}
Alexander Grothendieck, \emph{Sur quelques points d'alg\`ebre homologique},
  T\^ohoku Math. J. (2) \textbf{9} (1957), 119--221. \MR{0102537 (21 \#1328)}

\bibitem{MR0147925}
Alain Guichardet, \emph{Caract\`eres des alg\`ebres de {B}anach involutives},
  Ann. Inst. Fourier (Grenoble) \textbf{13} (1963), 1--81. \MR{0147925 (26
  \#5437)}

\bibitem{MR2806137}
Matthew~B. Hastings and Terry~A. Loring, \emph{Topological insulators and
  {$C^*$}-algebras: theory and numerical practice}, Ann. Physics \textbf{326}
  (2011), no.~7, 1699--1759. \MR{2806137 (2012j:82005)}

\bibitem{MR0358873}
Nigel Hitchin, \emph{Harmonic spinors}, Advances in Math. \textbf{14} (1974),
  1--55. \MR{MR0358873 (50 \#11332)}

\bibitem{MR0172132}
Lars Ingelstam, \emph{Real {B}anach algebras}, Ark. Mat. \textbf{5} (1964),
  239--270 (1964). \MR{0172132 (30 \#2358)}

\bibitem{MR585235}
V.~F.~R. Jones, \emph{A {${\rm II}_{1}$} factor anti-isomorphic to itself but
  without involutory antiautomorphisms}, Math. Scand. \textbf{46} (1980),
  no.~1, 103--117. \MR{585235 (82a:46075)}

\bibitem{MR2513335}
Max Karoubi, \emph{Twisted {$K$}-theory---old and new}, {$K$}-theory and
  noncommutative geometry, EMS Ser. Congr. Rep., Eur. Math. Soc., Z\"urich,
  2008, pp.~117--149. \MR{2513335 (2010h:19010)}

\bibitem{MR918241}
G.~G. Kasparov, \emph{Equivariant {$KK$}-theory and the {N}ovikov conjecture},
  Invent. Math. \textbf{91} (1988), no.~1, 147--201. \MR{918241 (88j:58123)}

\bibitem{Kitaev}
Alexei Kitaev, \emph{Periodic table for topological insulators and
  superconductors}, Landau Memorial Conference, AIP Conf. Proc., vol. 1134,
  2009, paper 22, arXiv:0901.2686.

\bibitem{MR891589}
Gottfried K{\"o}the, \emph{Stanis\l aw {M}azur's contributions to functional
  analysis}, Math. Ann. \textbf{277} (1987), no.~3, 489--528. \MR{891589
  (88i:01112)}

\bibitem{MR1914617}
Vincent Lafforgue, \emph{{$K$}-th\'eorie bivariante pour les alg\`ebres de
  {B}anach et conjecture de {B}aum-{C}onnes}, Invent. Math. \textbf{149}
  (2002), no.~1, 1--95. \MR{1914617 (2003d:19008)}

\bibitem{MR1995682}
Bingren Li, \emph{Real operator algebras}, World Scientific Publishing Co.,
  Inc., River Edge, NJ, 2003. \MR{1995682 (2004k:46100)}

\bibitem{MR0156292}
Andr{\'e} Lichnerowicz, \emph{Spineurs harmoniques}, C. R. Acad. Sci. Paris
  \textbf{257} (1963), 7--9. \MR{MR0156292 (27 \#6218)}

\bibitem{Mazet}
Pierre Mazet, \emph{La preuve originale de {S}. {M}azur pour son
  th{\'e}or{\`e}me sur les alg{\`e}bres norm{\'e}es}, Gaz. Math. (2007),
  no.~111, 5--11.

\bibitem{zbMATH02514879}
S.~{Mazur}, \emph{{Sur les anneaux lin\'eaires}}, {C. R. Acad. Sci., Paris}
  \textbf{207} (1938), 1025--1027 (French).

\bibitem{MR1606278}
Ruben Minasian and Gregory Moore, \emph{{$K$}-theory and {R}amond-{R}amond
  charge}, J. High Energy Phys. (1997), no.~11, Paper 2, 7 pp.\ (electronic).
  \MR{1606278 (2000a:81190)}

\bibitem{MR728142}
D.~R. Morrison, \emph{On {$K3$} surfaces with large {P}icard number}, Invent.
  Math. \textbf{75} (1984), no.~1, 105--121. \MR{728142 (85j:14071)}

\bibitem{MR3141812}
El-ka{\"{\i}}oum~M. Moutuou, \emph{On groupoids with involutions and their
  cohomology}, New York J. Math. \textbf{19} (2013), 729--792. \MR{3141812}

\bibitem{MR3158706}
\bysame, \emph{Graded {B}rauer groups of a groupoid with involution}, J. Funct.
  Anal. \textbf{266} (2014), no.~5, 2689--2739. \MR{3158706}

\bibitem{MR2137825}
Viacheslav~V. Nikulin and Sachiko Saito, \emph{Real {$K3$} surfaces with
  non-symplectic involution and applications}, Proc. London Math. Soc. (3)
  \textbf{90} (2005), no.~3, 591--654. \MR{2137825 (2006b:14063)}

\bibitem{MR0270162}
T.~W. Palmer, \emph{Real {$C\sp*$}-algebras}, Pacific J. Math. \textbf{35}
  (1970), 195--204. \MR{0270162 (42 \#5055)}

\bibitem{MR548006}
Gert~K. Pedersen, \emph{{$C^{\ast} $}-algebras and their automorphism groups},
  London Mathematical Society Monographs, vol.~14, Academic Press, Inc.
  [Harcourt Brace Jovanovich, Publishers], London-New York, 1979. \MR{548006
  (81e:46037)}

\bibitem{MR724030}
M.~G. Penington and R.~J. Plymen, \emph{The {D}irac operator and the principal
  series for complex semisimple {L}ie groups}, J. Funct. Anal. \textbf{53}
  (1983), no.~3, 269--286. \MR{724030 (85d:22016)}

\bibitem{MR2151029}
Joseph Polchinski, \emph{String theory. {V}ol. {I}}, Cambridge Monographs on
  Mathematical Physics, Cambridge University Press, Cambridge, 2005, An
  introduction to the bosonic string, Reprint of the 2003 edition. \MR{2151029
  (2006j:81149a)}

\bibitem{MR2151030}
\bysame, \emph{String theory. {V}ol. {II}}, Cambridge Monographs on
  Mathematical Physics, Cambridge University Press, Cambridge, 2005,
  Superstring theory and beyond, Reprint of 2003 edition. \MR{2151030
  (2006j:81149b)}

\bibitem{MR1634408}
Iain Raeburn and Dana~P. Williams, \emph{Morita equivalence and
  continuous-trace {$C^*$}-algebras}, Mathematical Surveys and Monographs,
  vol.~60, American Mathematical Society, Providence, RI, 1998. \MR{1634408
  (2000c:46108)}

\bibitem{MR679708}
Marc~A. Rieffel, \emph{Morita equivalence for operator algebras}, Operator
  algebras and applications, {P}art {I} ({K}ingston, {O}nt., 1980), Proc.
  Sympos. Pure Math., vol.~38, Amer. Math. Soc., Providence, R.I., 1982,
  pp.~285--298. \MR{679708 (84k:46045)}

\bibitem{MR679694}
Jonathan Rosenberg, \emph{Homological invariants of extensions of {$C^{\ast}
  $}-algebras}, Operator algebras and applications, {P}art 1 ({K}ingston,
  {O}nt., 1980), Proc. Sympos. Pure Math., vol.~38, Amer. Math. Soc.,
  Providence, RI, 1982, pp.~35--75. \MR{679694 (85h:46099)}

\bibitem{MR842428}
\bysame, \emph{{$C\sp \ast$}-algebras, positive scalar curvature, and the
  {N}ovikov conjecture. {III}}, Topology \textbf{25} (1986), no.~3, 319--336.
  \MR{MR842428 (88f:58141)}

\bibitem{MR1018964}
\bysame, \emph{Continuous-trace algebras from the bundle theoretic point of
  view}, J. Austral. Math. Soc. Ser. A \textbf{47} (1989), no.~3, 368--381.
  \MR{1018964 (91d:46090)}

\bibitem{MR1133900}
\bysame, \emph{The {$K{\rm O}$}-assembly map and positive scalar curvature},
  Algebraic topology Pozna\'n 1989, Lecture Notes in Math., vol. 1474,
  Springer, Berlin, 1991, pp.~170--182. \MR{MR1133900 (92m:53060)}

\bibitem{MR3305978}
\bysame, \emph{Real {B}aum-{C}onnes assembly and {T}-duality for torus
  orientifolds}, J. Geom. Phys. \textbf{89} (2015), 24--31. \MR{3305978}

\bibitem{MR1321004}
Jonathan Rosenberg and Stephan Stolz, \emph{A ``stable'' version of the
  {G}romov-{L}awson conjecture}, The \v Cech centennial (Boston, MA, 1993),
  Contemp. Math., vol. 181, Amer. Math. Soc., Providence, RI, 1995,
  pp.~405--418. \MR{MR1321004 (96m:53042)}

\bibitem{MR1490835}
Sh{\^o}ichir{\^o} Sakai, \emph{{$C^*$}-algebras and {$W^*$}-algebras}, Classics
  in Mathematics, Springer-Verlag, Berlin, 1998, Reprint of the 1971 edition.
  \MR{1490835 (98k:46085)}

\bibitem{MR2077669}
Thomas Schick, \emph{Real versus complex {$K$}-theory using {K}asparov's
  bivariant {$KK$}-theory}, Algebr. Geom. Topol. \textbf{4} (2004), 333--346.
  \MR{2077669 (2005f:19007)}

\bibitem{MR1267059}
Herbert Schr{\"o}der, \emph{{$K$}-theory for real {$C^*$}-algebras and
  applications}, Pitman Research Notes in Mathematics Series, vol. 290, Longman
  Scientific \& Technical, Harlow; copublished in the United States with John
  Wiley \& Sons, Inc., New York, 1993. \MR{1267059 (95f:19006)}

\bibitem{MR0450380}
Jean-Pierre Serre, \emph{Linear representations of finite groups},
  Springer-Verlag, New York-Heidelberg, 1977, Translated from the second French
  edition by Leonard L. Scott, Graduate Texts in Mathematics, Vol. 42.
  \MR{0450380 (56 \#8675)}

\bibitem{MR1403963}
Stephan Stolz, \emph{Positive scalar curvature metrics---existence and
  classification questions}, Proceedings of the International Congress of
  Mathematicians, Vol.\ 1, 2 (Z\"urich, 1994) (Basel), Birkh\"auser, 1995,
  pp.~625--636. \MR{MR1403963 (98h:53063)}

\bibitem{MR1937026}
\bysame, \emph{Manifolds of positive scalar curvature}, Topology of
  high-dimensional manifolds, No. 1, 2 (Trieste, 2001), ICTP Lect. Notes,
  vol.~9, Abdus Salam Int. Cent. Theoret. Phys., Trieste, 2002, Papers from the
  School on High-Dimensional Manifold Topology held in Trieste, May 21--June 8,
  2001, Available electronically at {\tt
  http://publications.ictp.it/lns/vol9.html}, pp.~661--709. \MR{MR1937026
  (2003m:53059)}

\bibitem{MR563372}
Erling St{\o}rmer, \emph{Real structure in the hyperfinite factor}, Duke Math.
  J. \textbf{47} (1980), no.~1, 145--153. \MR{563372 (81g:46088)}

\bibitem{MR2231869}
Jean-Louis Tu, \emph{Groupoid cohomology and extensions}, Trans. Amer. Math.
  Soc. \textbf{358} (2006), no.~11, 4721--4747 (electronic). \MR{2231869
  (2007i:22008)}

\bibitem{MR2274533}
Bert van Geemen and Alessandra Sarti, \emph{Nikulin involutions on {$K3$}
  surfaces}, Math. Z. \textbf{255} (2007), no.~4, 731--753. \MR{2274533
  (2007j:14057)}

\bibitem{MR0498999}
Garth Warner, \emph{Harmonic analysis on semi-simple {L}ie groups. {I}},
  Springer-Verlag, New York-Heidelberg, 1972, Die Grundlehren der
  mathematischen Wissenschaften, Band 188. \MR{0498999 (58 \#16979)}

\bibitem{MR894996}
Antony Wassermann, \emph{Une d\'emonstration de la conjecture de
  {C}onnes-{K}asparov pour les groupes de {L}ie lin\'eaires connexes
  r\'eductifs}, C. R. Acad. Sci. Paris S\'er. I Math. \textbf{304} (1987),
  no.~18, 559--562. \MR{894996 (89a:22010)}

\bibitem{MR1674715}
Edward Witten, \emph{D-branes and {$K$}-theory}, J. High Energy Phys. (1998),
  no.~12, Paper 19, 41 pp.\ (electronic). \MR{1674715 (2000e:81151)}

\bibitem{MR0448079}
Wies{\l}aw {\.Z}elazko, \emph{Banach algebras}, Elsevier Publishing Co.,
  Amsterdam-London-New York; PWN--Polish Scientific Publishers, Warsaw, 1973,
  Translated from the Polish by Marcin E. Kuczma. \MR{0448079 (56 \#6389)}

\end{thebibliography}
\end{document}